\documentclass[12pt]{amsart}
\usepackage[utf8]{inputenc}
\usepackage{verbatim}
\usepackage{amsmath, amsthm,amssymb,bbm,enumerate,mathrsfs,mathtools}

\title{Analysis of the Lifting Graph}
\author[A.~Assem]{Amena Assem}
\address{University of Waterloo, ON, Canada}
\email{a36mahmo@uwaterloo.ca}

\date{}

\keywords{edge-connectivity, lifting, splitting off, connectivity augmentation, network design. This research was supported in part by the Natural Sciences and Engineering Research Council of Canada (NSERC), funding reference number RGPIN-2019-04669.}


\usepackage{xcolor} 	
\usepackage[unicode]{hyperref}
\hypersetup{
	colorlinks,
    linkcolor={red!60!black},
    citecolor={green!60!black},
    urlcolor={blue!60!black},
}
\usepackage[abbrev, msc-links]{amsrefs}

\usepackage{geometry}
\theoremstyle{plain}
\newtheorem{theorem}{Theorem}[section]

\newtheorem{proposition}[theorem]{Proposition}
\newtheorem{claim}[theorem]{Claim}
\newtheorem{corollary}[theorem]{Corollary}
\newtheorem{lemma}[theorem]{Lemma}

\theoremstyle{definition}
\newtheorem{remark}[theorem]{Remark}
\newtheorem{definition}[theorem]{Definition}

\begin{document}

{\centering\footnotesize \textit{This paper is dedicated to my father, the architect of my mathematical path.}\par}

\begin{abstract}
The \textit{lifting} or \textit{splitting-off} operation on graphs is performed by deleting two edges $sv$ and $sw$ having a common end $s$ and adding a new edge between $v$ and $w$. Such a lift is considered good if it preserves a certain local edge-connectivity between the pairs of vertices different from the vertex $s$ at which lifting takes place. The operation is important for inductive proofs concerning edge-connectivity, and can be seen widely applied in the literature on connectivity augmentation, network design, orientation (of finite and infinite graphs), and edge-disjoint linkage. It was studied by Lov\'asz, who used the term splitting-off, and Mader, who used the term lifting. They proved the first two significant results on it, in 1976 and 1978 respectively, showing the existence of a good lift under certain conditions. Then it was used and studied by other researchers, through the 1980s, both for undirected and directed graphs. In particular, it was investigated further by Frank who proved in 1992 that there are $\lfloor \deg(s)/2 \rfloor $ disjoint good lifts. Motivated by the applications, a new method for studying the operation was introduced by Jord\'an in the late 1990s. He defined and studied the structure of the \textit{non-admissibility graph}, which is the complement of the lifting graph; the subject of this paper. He proved a number of significant structural results on it, which he applied to connectivity augmentation. Independently, in 2016, Thomassen defined the \textit{lifting graph}, and called its complement the \textit{bad graph}, to apply it in finding orientations of infinite graphs. Later in the same year, Thomassen with Ok and Richter extended the study, and applied their results to linkages in infinite graphs.

Here we give a more comprehensive analysis of the structure of the lifting graph. Our approach and results, found independently, have a lot of similarities with those of Jord\'an, who assumed that $\deg(s)$ is even for the most part, and we build on the results by Ok, Richter, and Thomassen. The \emph{lifting graph} $L(G,s,k)$ has as its vertices the set of edges incident with $s$, where a pair is adjacent iff local $k$-edge-connectivity is preserved after lifting it. We find here the structure of $L(G,s,k)$ in the different possibilities for the parities of $\deg(s)$ and $k$. Our main new contribution is establishing a clear correspondence between the structural arrangements of the maximal independent sets of $L(G,s,k)$ and the \emph{dangerous} sets, which define near-critical cuts in $G$. An important consequence is the description of the structure of $G$ from the structure of $L(G,s,k)$ when the latter has a connected complement. In that case, $G-s$ has a path-like or cycle-like structure around $s$, which was previously found by Jord\'an but only in the case when $\deg(s)$ is even. Although dangerous sets were used by Jord\'an, and independent sets were studied when $\deg(s)$ is small by Ok, Richter, and Thomassen, the correspondence between the two as displayed in the structures of $G$ versus $L(G,s,k)$ was not clearly made to stand out. Results of this paper have been applied to solve questions on orientations and linkages in infinite graphs.  
\end{abstract}

\maketitle

\section{Introduction}

The graphs in this paper are finite loopless multigraphs. We present the main results of this paper at the end of this section in Theorem \ref{new results}, but first we pave the way to it through the following introduction.

\emph{Lifting}, also known as \emph{splitting-off}, is an operation widely used in the study of edge-connectivity, and is particularly important in inductive proofs and algorithms. It has applications in network design problems as can be seen in \cite{LapChiLau-Thesis} and \cite{NetworkDesign}, and in connectivity augmentation, which was extensively studied by Frank, Jord\'an, and Bang-Jensen, among others, who proved many theorems on splitting-off for the purpose of that application \cite{Frank-augmenting-meet-requirements-1992}, \cite{Frank-augmentation-network-design-1994}, \cite{Bang-Jensen-Gabow-Jordan-Szigeti1999}, \cite{Jordan1999ConstrainedSplitting}, \cite{Bang-Jensen2000splitting-specified-subset}, \cite{Bang-Jensen2004splitting-between-two-subset}. Other areas of application include finding orientations of high arc-connectivity \cite{mader1978reduction}, \cite{thomassen2016orientations}, \cite{assem2023towards}, \cite{Max2023orientation}, and edge-disjoint linkages \cite{Huck}, \cite{ORT2016linkages}, both in finite and infinite graphs. A variant of the operation for directed, and mixed, graphs was also studied \cite{Bang-JensenFrankJackson-splitting-for-mixed-graphs}. More can be learned about lifting, or splitting-off, and its applications, which also include the constructive characterizations of certain graph classes, from Chapter 8 in Frank's book \emph{Connections in Combinatorial Optimization} \cite{frank2011connections}. Also, in the same book, and in \cite{Frank-submodular} and \cite{BernathKiraly2008}, a matroid-theoretic perspective on the connection between the operation and submodular and skew supermodular funtions is presented.

We point out here that the term \emph{splitting-off} is more common than \emph{lifting} in referring to this operation. Almost the entire literature on the operation and its applications from the 1970s until 2015, with the exception of the paper by Mader \cite{mader1978reduction}, uses the term \emph{splitting-off}. In 1978, Mader called it \emph{lifting} in his paper, which was translated from German to English by Bollob\'as. In 2016, Thomassen used the same term as Mader in an application of lifting to finding orientations in infinite graphs \cite{thomassen2016orientations}, and so did he again in the same year in his subsequent result on linkages with Ok and Richter \cite{ORT2016linkages}. The work in the current paper is built on that of Thomassen, Ok, and Richter, and was developed for the same purpose of extending the results on linkages and orientations in infinite graphs. Some of these applications can be seen in other papers by the author of this paper \cite{assem2023towards}, \cite{Max2023orientation}. Therefore, we use the same term as Thomassen, i.e. \emph{lifting}.

If $s$ is a vertex in a graph $G$ and $sv$ and $sw$ are two edges incident with $s$, then to \textit{lift} the pair of edges $sv$ and $sw$ is to delete them and add the edge $vw$ (if $v=w$, then this results in a loop, below we show how to avoid this). The resulting graph $G_{v,w}$ is the \textit{lift of $G$ at $sv$ and $sw$}, or simply a \textit{lift at} $s$. A lift of $G$ at $s$ is \textit{feasible} if the number of edge-disjoint paths between any two vertices other than $s$ stays the same after the lift. We also say in that case that $sv$ and $sw$ form a \textit{feasible pair}. If $sv$ and $sw$ are parallel edges, then note that they form a feasible pair if and only if their deletion, without adding the loop, preserves the number of edge-disjoint paths between every two vertices different from $s$. Therefore, we redefine \textit{lifting a pair of parallel edges as only deleting them}. In this way the graph remains loopless after lifting.

Unlike contraction, a set of $n$ edge-disjoint paths in a graph obtained from a sequence of feasible lifts can always be returned to a set of $n$ edge-disjoint paths in the original graph, and this is very useful in applications. Also lifting does not affect the parity of an edge-cut.

Lov\'asz and Mader proved the first two significant results on lifting. However, they both referred to a paper by Kotzig on splittings preserving the edge-connectivity of $4$-regular multigraphs \cite{Kotzig}. In 1976 Lov\'asz proved that for any vertex $s$ in an Eulerian graph $G$, a feasible lift exists of $G$ at $s$ \cite{lovasz1976eulerian}. In fact, he proved something stronger \cite[Theorem~1]{lovasz1976eulerian}, that in an Eulerian graph, for any edge $sv$ incident with $s$, there is another edge $sw$ incident with $s$ such that lifting (which he called splitting off) $sv$ and $sw$ does not affect the local edge-connectivity. Note that this, and the fact that lifting does not affect the parity of the degree of any vertex, means that in an Eulerian graph one can perform a complete sequence of lifts at any given vertex $s$ until there are no edges incident with it, and then delete the remaining isolated vertex and have a smaller graph that still has the same local edge-connectivities.

Mader defined lifting only for pairs of edges $sv$ and $sw$ such that $v\neq w$, and in 1978 \cite{mader1978reduction} he proved that for any graph $G$, not necessarily Eulerian, if the vertex $s$ is not a cut-vertex, $\deg(s) \geq 4$, and $s$ has at least two neighbours, then there exists a feasible lift of $G$ at $s$ of a pair of edges with different other ends. In the same paper \cite[Theorem~22]{mader1978reduction} Mader used lifting to give an alternative proof for the theorem of Nash-Williams that every finite graph has an admissible orientation \cite{nash1960orientations}, where an admissible orientation is one such that for any two vertices $x$ and $y$ the number of edge-disjoint paths directed from $x$ to $y$ is at least half the number of edge-disjoint paths between $x$ and $y$, rounded down.

When $\deg(s)=2$, the pair of edges incident with $s$ is clearly feasible, regardless of whether they are parallel or whether $s$ is a cut-vertex or not. The case when $\deg(s)=3$ is problematic as it is possible that no pair of edges incident with $s$ forms a feasible lift, for example if $G$ is $K_4$ or $K_{4,1}$. Section 7 is dedicated to studying that case. 

In 1992 Frank \cite{frank} proved that if $s$ is not incident with a cut-edge and $\deg(s)\neq 3$, then there are $\lfloor \deg(s)/2 \rfloor$ disjoint feasible pairs of edges. In his paper, unlike Mader, he allowed lifting parallel pairs of edges. Frank's result does not necessarily mean that these disjoint pairs are feasible if lifted in sequence. If $sv_1$, $sw_1$, and $sv_2$, $sw_2$ are two disjoint feasible pairs, then it is possible that the pair $sv_2$, $sw_2$ is not feasible in $G_{v_1,w_1}$ (see Figure \ref{example 2 lifting from deg 5 to deg 3}). Frank's theorem only tells us that there is at most one edge incident with $s$ that does not form a feasible pair with any other edge incident with $s$. In other words, any independent set in $L(G,s,k)$ has size at most  $\lceil \deg(s)/2 \rceil $ (recall that a set of vertices is \textit{independent} if no two vertices in it are adjacent).

\begin{theorem}\label{Frank}\cite{frank}
(Frank 1992) Let $G$ be a graph containing a vertex $s$ not incident with a cut-edge such that $\deg(s)\neq 3$. Then there are $\lfloor \deg(s)/2 \rfloor $ pairwise disjoint pairs of edges incident with $s$ each of which defining a feasible lift.
\end{theorem}

In this paper, we are only concerned with preserving a prescribed uniform local edge-connectivity $k$, which may be less than the number of edge-disjoint paths between one pair of vertices. We only need the local edge-connectivity between any two vertices, other than the fixed vertex $s$ at which the lifting is performed, not to fall below the target connectivity $k$ after lifting. The $k$-\textit{lifting graph}, to be defined below, encodes the information of which pairs of edges satisfy this. 

Ideally, in applications we want to have all but at most one of the edges incident with $s$ lifted. The final graph, resulting from a sequence of lifts, is $k$-edge-connected because no path between two vertices other than $s$ can go through $s$ when $\deg(s)\leq 1$. Then we can apply theorems that hold for $k$-edge-connected graphs to that final graph. Ending the sequence when all the remaining edges incident with $s$ are parallel will also satisfy the same purpose. Such a sequence of lifts is not always possible when $\deg(s)$ is odd as we may have to stop at $\deg(s)=3$ as will be seen in some examples.

Motivated by the applications, Jord\'an in 1999 started a new way of studying lifting in order to see the structure induced by the non-admissible pairs of edges (whose lifting does not preserve a uniform local edge connectivity $k$) incident with a vertex $s$. His results were published in 2003 in \cite{Jordan1999ConstrainedSplitting}. He defined the \emph{non-admissibility} graph $B(s)$ on the set of neighbours $N(s)$ of $s$ in $G$. Two vertices $u,v\in N(s)$ are adjacent in $B(s)$ if and only if the pair $su, sv$ is not-admissible in $G$. He proved the following structural results:

\begin{itemize}
\item If $k$ and $\deg(s)$ are both even and $|N(s)|\geq 2$, then $B(s)$ is the disjoint union of at least two complete graphs \cite[Theorem~3.2]{Jordan1999ConstrainedSplitting} (compare this to (4) in Theorem \ref{ORT} below).

\item If $\deg(s)$ is even and $|N(s)\geq 2|$, then every component of $B(s)$ is one of the following: a clique, two cliques meeting at one vertex, two disjoint cliques connected by a path, or a cycle of length at least $4$ \cite[Theorem~3.4]{Jordan1999ConstrainedSplitting}.

\item If $\deg(s)$ is even, $k\geq 2$, and $|N(s)\geq 2|$, then $B(s)$ is $2$-edge-connected if and only if $B(s)$ is a cycle of length $\deg(s)$, $k$ is odd, and $G$ has a \emph{round} structure with respect to $s$ \cite[Theorem~3.6]{Jordan1999ConstrainedSplitting}.
\end{itemize}

In 2016, independently and using a different method, Thomassen showed that the $k$-lifting graph of an Eulerian graph has a disconnected complement \cite[Theorem~2]{thomassen2016orientations}. Note that this is weaker than what Jord\'an proved, because it did not specify the structure of each component in the complement (a clique), and because Jord\'an assumed that only $\deg(s)$ is even but the graph is not necessarily assumed to be Eulerian. Then Ok, Richter, and Thomassen, also in 2016, showed that when both $\deg(s)$ and $k$ are even, the lifting graph is a complete multipartite graph \cite[Theorem~1.2]{ORT2016linkages} (also see (4) in Theorem \ref{ORT} below), hence its complement is the disjoint union of cliques as in Jord\'an's result (\cite[Theorem~3.2]{Jordan1999ConstrainedSplitting}). They proved this independently from Jord\'an's work and with a different proof idea, however they too made use of dangerous sets as Jord\'an did, which is a standard tool in this context. They also described to some detail the structure of the lifting graph in the case when it is disconnected. Their results are listed in Theorem \ref{ORT} below, which we state after introducing the necessary technical definitions.

Thomassen used his result on the lifting graph to prove that an edge-connectivity of $8k$ is sufficient for an infinite graph to admit a $k$-arc-connected orientation \cite{thomassen2016orientations}. It was not known before whether there is any function $f(k)$ at all such that $f(k)$-edge-connected infinite graphs have a $k$-arc-connected orientation. Theorem \ref{ORT} was crucial in Ok, Richter, and Thomassen's proof that, for odd $k$, every $(k+2)$-edge-connected, locally-finite, $1$-ended infinite graph is weakly $k$-linked \cite{ORT2016linkages}. This was an important step towards Thomassen's weak linkage conjecture for infinite graphs \cite{Thomassen2linked}. It brought us closer to extending from finite to infinite graphs the theorem of Huck \cite{Huck} that for odd $k$ an edge-connectivity of $(k+1)$ is enough for a finite graph to be weakly $k$-linked, which they used in their proof. That result of Huck for finite graphs was also proved using lifting. The author of this paper extended the linkage result of Ok, Richter, and Thomassen and showed that Huck's theorem holds for infinite graphs, i.e. an improvement from $(k+2)$ to $(k+1)$ in the edge-connectivity assumption. She used results on the lifting graph from the current paper for this extension. The paper on linkage is not yet written. Now we present some definitions needed to state the new results.

\begin{definition} (Connectivity ignoring $s$)
A graph $G$ is $(s,k)$-edge-connected if it is a graph on at least three vertices, $s$ is a vertex of $G$, $k$ is a positive integer, and any two vertices in $G$ both different from $s$ are joined by $k$ pairwise edge-disjoint paths in $G$ (that may go through $s$).
\end{definition}

 Note that $G$ is not necessarily $k$-edge-connected in this definition. By Menger's Theorem \ref{Menger}, the definition is equivalent to that for any subset $X$ of vertices such that $\emptyset \neq X \subsetneq V(G)\setminus \{s\}$, we have $|\delta(X)|\geq k$, where $\delta(X)$ denotes the set of edges with one end in $X$ and one end outside $X$. 
 
 If $k\geq 2$, $\deg(s)\geq 2$, and $G$ is $(s,k)$-edge-connected then $s$ is not incident with a cut-edge, which is a condition we need in almost all of our statements. Thus, the assumptions that $k\geq 2$ and $\deg(s)\geq 2$ will be found frequently throughout the paper. We can now talk about lifting with a prescribed edge-connectivity using this definition.

 Let $G$ be an $(s,k)$-edge-connected graph, then the edges $sv$ and $sw$, are $k$-\textit{liftable}, or they form a $k$-\textit{liftable} pair, if $G_{v,w}$ is $(s,k)$-edge-connected. We simply say \textit{liftable} when connectivity is understood from the context. We also say that the edge $sv$ \textit{lifts} with $sw$. The $k$-\textit{lifting} graph $L(G,s,k)$ has as its vertices the edges of $G$ incident with $s$ and two vertices are adjacent in $L(G,s,k)$ if they form a $k$-liftable pair. 

 Note that if $k=1$ then being $1$-liftable means having one path between every two non-$s$ vertices after lifting. Thus, a pair of edges $sv$ and $sw$ is not $1$-liftable if and only if $\deg(s)\geq 3$ and $\{sv,sw\}$ is a minimal cut or $sv$ and $sw$ are both cut-edges. This can be seen by looking at the components of $G-s$. 

 Thus the lifting graph in case $k=1$ is $K_2$ if $\deg(s)=2$, but if $\deg(s)\geq 3$, then it is the union of two graphs. One is a complete multipartite graph with one cluster containing all the cut-edges incident with $s$ (this could possibly be the only cluster) and whose other clusters are the minimal cuts of size at least $2$ consisting of edges incident with $s$. The other is a union of cliques on the clusters of size at least $3$ of the first graph that are different from the one cluster containing all the cut edges. Or in other words, it is a complete graph minus one clique and some disjoint edges. This is the characterization of liftability in the case $k=1$.

Before presenting our theorem, we present here first the theorem by Ok, Richter, and Thomassen for comparison.

\begin{theorem}\label{ORT}\cite[Theorem~1.2]{ORT2016linkages} (Ok, Richter, and Thomassen 2016)
Let $G$ be an $(s,k)$-edge-connected graph such that $\deg(s)\geq 4$ and $s$ is not incident with a cut-edge. Then:
\begin{enumerate}
\item[(1)] The $k$-lifting graph $L(G,s,k)$ has at most two components;
\item[(2)] If $\deg(s)$ is odd and $L(G,s,k)$ has two components, then one has only one vertex and the other component is complete multipartite;
\item[(3)] If $\deg(s)$ is even and $L(G,s,k)$ has two components, then each component is complete multipartite with an even number of vertices and;
\item[(4)] If $\deg(s)$ and $k$ are both even, then $L(G,s,k)$ is a connected, complete, multipartite graph (in particular, it has a disconnected complement).
\end{enumerate}
If either $L(G,s,k)$ is not connected or both $\deg(s)$ and $k$ are even, then any component of $L(G,s,k)$ with at least $4$ vertices is not a star $K_{1,r}$.
\end{theorem}

Note that point (4) was previously proved by Jord\'an in \cite[Theorem~3.2]{Jordan1999ConstrainedSplitting}. In this paper we strengthen Theorem \ref{ORT} and we also extend Jord\'an's results to the case when $\deg(s)$ is odd and present both in a more general setting. We give a description of the lifting graph for the different possibilities of the parities of $\deg(s)$ and $k$. In particular, the interesting case of a connected complement of the lifting graph is completely described: the maximal independent sets of the lifting graph form either a path or a cycle via intersections. The same (path or cycle) structure is displayed in the graph, with blobs corresponding to the intersections. Consecutive blobs are joined by precisely $(k-1)/2$ edges except in the case when the structure is a path of length $2$ of blobs. The cycle structure is the same as the \emph{round} structure found by Jord\'an in \cite[Theorem~3.6]{Jordan1999ConstrainedSplitting}. To describe these path and cycle structures formally, we introduce the following definitions.

\begin{figure}[!h]
  \centering
  \begin{minipage}[b]{0.4\textwidth}
    \includegraphics[width=\textwidth]{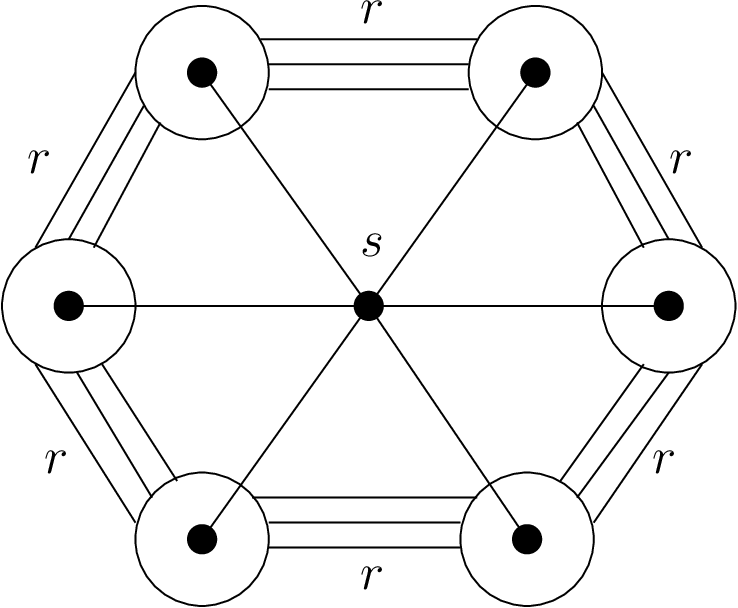}
    \caption {cycle structure}
    \label{cycle}
  \end{minipage}
  \hfill
  \begin{minipage}[b]{0.5\textwidth}
    \includegraphics[width=\textwidth]{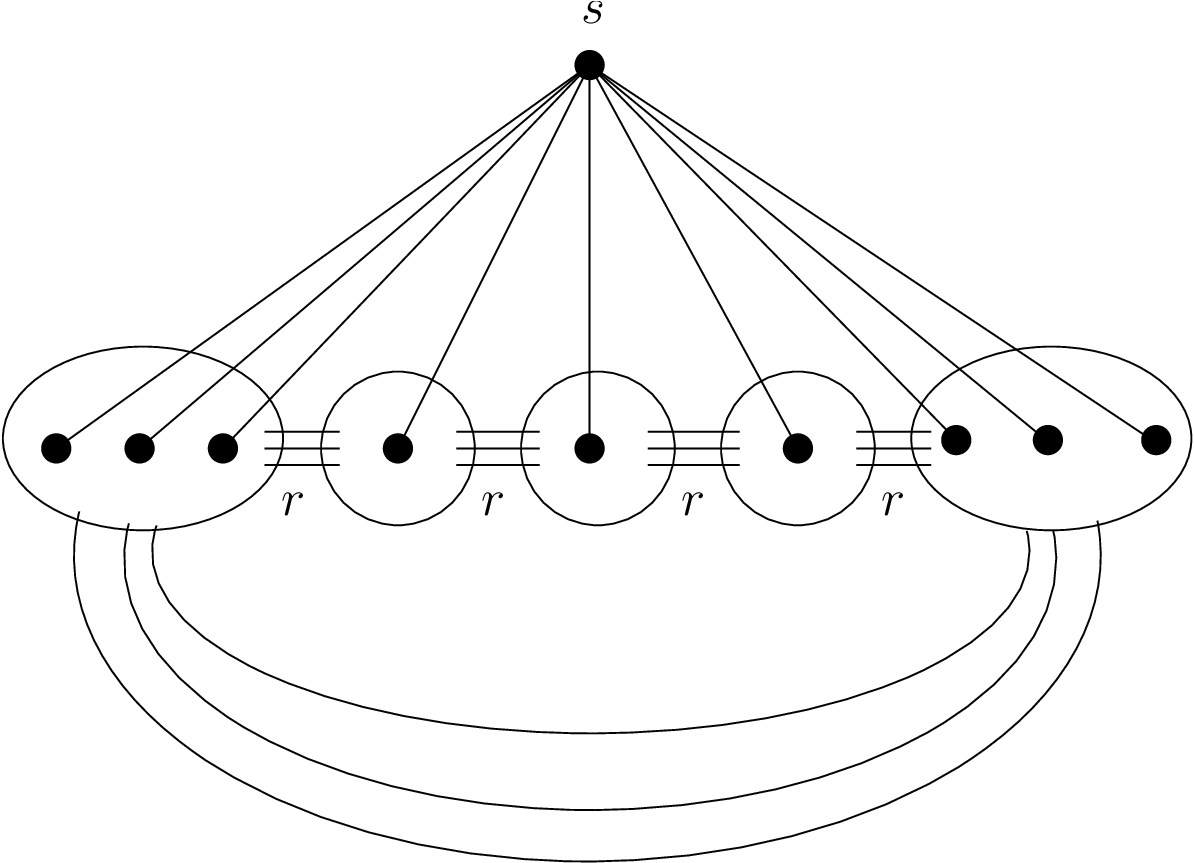}
    \caption{path structure}
    \label{path}
  \end{minipage}
\end{figure}

\begin{definition} (Path or Cycle Structure)
 Let $G$ be a graph, $s$ a vertex of $G$, $r\geq 1$ and $n\geq 3$ integers, such that $G-s$ has pairwise disjoint subgraphs $H_0,\dots,H_{n-1}$ satisfying $V(G-s)=\bigcup_{i=0}^{n-1}V(H_i)$. The graph $G$ has an \emph{$(s,r)$-cycle structure} (\emph{$(s,r)$-path structure}) if for $i=0,\dots,n-1$ ($i=1,\dots,n-2$), $\delta(H_i)$ consists of $2r+1$ edges: one edge whose other end is $s$, and $r$ edges with other ends in $H_{i-1}$, and another $r$ edges with other ends in $H_{i+1}$, where indices are read modulo $n$. The $H_i$ are the \emph{blobs} of the structure. See Figures \ref{cycle} and \ref{path}.
\end{definition}

We summarize the main new results in the following theorem. Here we assume that $\deg(s)> 4$ as the case $\deg(s)=4$ was completely described by Ok, Richter, and Thomassen in \cite{ORT2016linkages}, and will be presented in Lemma \ref{4 cycle} ($L(G,s,k)$ is one of $K_2\cup K_2$, $C_4$, or $K_4$) for completeness.

\begin{theorem}(Main Results)\label{new results}
Let $G$ be an $(s,k)$-edge-connected graph such that $k\geq 2$ and $\deg(s)> 4$. The following holds: 
\begin{enumerate}
\item[(i)] If $L(G,s,k)$ is disconnected, then $\deg(s)$ is odd and $L(G,s,k)$ is an isolated vertex plus a balanced complete bipartite graph (Corollary \ref{the only disconnected case of lifting graph}).
\item[(ii)] If $k$ is even, then either $L(G,s,k)$ is complete multipartite or an isolated vertex plus a balanced complete bipartite graph (This follows from (3) and (4) in Theorem \ref{summary of lifting graph facts}).
\item[(iii)] If the complement of $L(G,s,k)$ is connected, then $k$ is odd except possibly in the case when $L(G,s,k)$ is as in (i) where it can be even as well as odd, and the complement is one of the following:
\begin{itemize}
\item a Hamilton cycle, or 
\item two cliques of the same size joined by a path (possibly a path of one vertex as in (i)).
\end{itemize}

The same cycle or path structure is displayed in $G$: 
\begin{itemize}
    \item $G$ has an \emph{$(s,r)$-cycle structure}, or
    \item an \emph{$(s,r)$-path structure}
\end{itemize}
with $r=(k-1)/2$ except possibly in the case when $L(G,s,k)$ is as in (i) where we only know that $r\geq (k-1)/2$. 
Moreover, each of the blobs of the structure is $\lceil k/2 \rceil$-edge-connected except possibly for the middle blob in the case when $L(G,s,k)$ is as in $(i)$ (Theorems \ref{connected complement}, \ref{cycle or path structure of G}, and Corollary \ref{the blobs are (k+1)/2 connected}).
\end{enumerate}
\end{theorem}

From $(i)$ we now know that $(3)$ in Theorem \ref{ORT} is not realizable when $\deg(s)>4$, and that the multipartite component in point $(2)$ of the same theorem is in fact bipartite (complete and balanced). If $\deg(s)$ and $k$ are both even, we know from \cite[Theorem~3.2]{Jordan1999ConstrainedSplitting} and $(4)$ in Theorem \ref{ORT} that $L(G,s,k)$ is complete multipartite (disconnected complement). If $\deg(s)$ is odd and $k$ is even, it is still possible that $L(G,s,k)$ is complete multipartite, however there is also a second possibility, as presented in (ii), that of an isolated vertex plus a balanced complete bipartite graph, seen differently, the complement is two cliques meeting at a vertex (disconnected lifting graph with connected complement). If $k$ is odd, then it is possible that both $L(G,s,k)$ and its complement are connected, in that case the complement is a Hamilton cycle or two cliques of the same size joined by a path containing at least one edge. These two structures can happen when $\deg(s)$ is even as well as when it is odd, but the connectivity $k$ must be odd. In \cite[Theorem~3.4]{Jordan1999ConstrainedSplitting} Jord\'an showed that when $\deg(s)$ is even each component of the complement of $L(G,s,k)$ is either a clique, two cliques meeting at a vertex, two cliques joined by a nontrivial path, or a cycle of length at least $4$. That is comparable to Corollary \ref{component of complement} here, which was a step towards proving (iii) above, holding regardless of the parity of $\deg(s)$. When $\deg(s)$ is even, Jord\'an also found the parameter $\frac{(k-1)}{2}$ for intersecting dangerous sets in \cite[Lemma~2.3]{Jordan1999ConstrainedSplitting}, and showed in \cite[Theorem~3.6]{Jordan1999ConstrainedSplitting} that the complement of $L(G,s,k)$ is $2$-edge-connected if and only if it is a cycle and $G$ has a \emph{round} structure with respect to $s$; the same as our cycle structure. Here we precisely describe the path and cycle structures regardless of the parity of $\deg(s)$.

Part (ii) of the above theorem was crucial to our applications in orientations of infinite graphs. In \cite{assem2023towards} we used the statement (ii) as it is to prove that $4k$-edge-connected, $1$-ended, locally finite graphs admit a $k$-arc-connected orientation. However, we proved a more general version of (ii) in \cite[Thoerem~3.3]{Max2023orientation} which states that if $k$ is even and $A$ is a proper subset of $V(G)$ such that between any two vertices in $A$ there are $k$-edge-disjoint paths in $G$, and if there is no edge with both end-vertices outside $A$, then for any $s\notin A$ such that $\deg(s)>3$, the lifting graph at $s$ is either complete multipartite, or consists of an isolated vertex and a balanced complete bipartite graph, if our target is to only preserve $k$-edge-connectivity between vertices of $A$. We used this latter result to prove that the theorem of Nash-Williams, that $2k$-edge-connectivity suffices for a $k$-arc-connected orientation, holds for locally finite graphs with countably many ends.

\noindent \textbf{\large Outline of the paper:} In Section 2 we present the standard definition of a dangerous set, which is the main object used in the proofs, and some essential lemmas about it. Section 3 presents some new results on the structure of the lifting graph. In Section 4 we study the structure more carefully via its maximal independent sets and corresponding dangerous sets, and in Section 5 we focus on the complement of the lifting graph. Section 6 presents a particularly important consequence of the previous results, that is the description of the structure of a graph from the structure of its lifting graph when the complement of the lifting graph is connected. Finally, Section 7 is dedicated to the study of the lifting graph when $\deg(s)=3$.

\section{Dangerous sets}

\begin{par}
Ok, Richter, and Thomassen proved Theorem \ref{ORT} by induction on $\deg(s)$, and the proof of the base cases, as will the proof of our extension of this theorem, relied on the concept of a \textit{dangerous set}, and a standard equation \ref{intersection of two cuts} about two intersecting edge-cuts. Recall that an \textit{edge-cut}, or simply a \textit{cut}, in a graph $G$ consists of, for some partition $(A, B)$ of $V(G)$, the set $\delta_G(A)$ (or simply $\delta(A)$) of all edges having one end in $A$ and one end in $B$.
\end{par}

\begin{figure}[!h]
     \centering
     \includegraphics[scale=0.75]{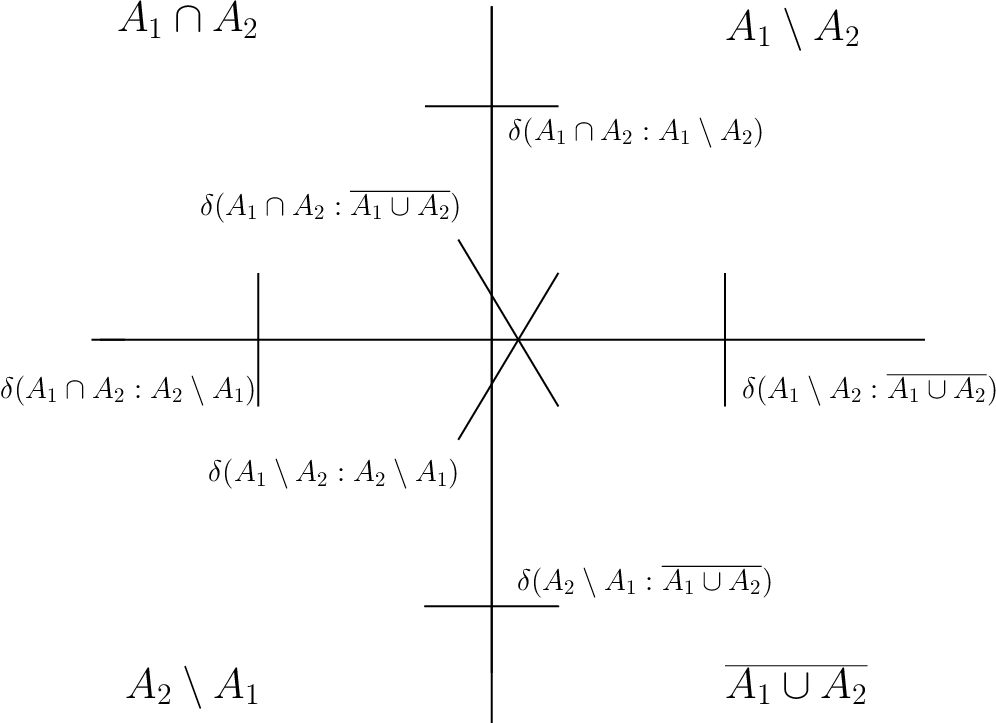}
     \caption{Two intersecting cuts.}
     \label{two cuts}
\end{figure}

\begin{par}
If $G$ is an $(s,k)$-edge-connected graph, then a subset $A$ of $V(G)\setminus \{s\}$ is \textit{$k$-dangerous}, or simply \textit{dangerous}, if $A\neq \emptyset$, $V(G)\setminus (A\cup\{s\})\neq \emptyset$, and $|\delta_G(A)|\leq k+1$. Dangerous sets are a standard tool in studying lifting and edge-connectivity. They can be seen in several papers by Frank and Jord\'an on splitting-off and connectivity augmentation, and also more propositions about them were introduced \cite{NetworkDesign} to prove results about degree bounded network designs with metric costs. 
\end{par}

\begin{par}
We will be dealing a lot with edge-cuts and so will be using the following version of Menger's Theorem \cite{Menger}. 
\end{par}

\begin{theorem}(Menger's Theorem)\label{Menger}\cite[Corollary~3.3.5(ii)]{Diestel}
Let $G$ be a graph and $x$ and $y$ two vertices in $G$. The maximum number of edge-disjoint paths between $x$ and $y$ equals the minimum size of an edge-cut separating $x$ from $y$ in $G$.
\end{theorem}

\begin{par}
\noindent The following standard equation, about two intersecting cuts $\delta(A_1)$ and $\delta(A_2)$, is essential to our proofs. For a vertex set $A$, let $\overline{A}:= V(G)\setminus A$. If $B$ is another set of vertices, then the set of edges with one end in $A$ and one end in $B$ is denoted by $\delta(A:B)$. Rudimentary counting, with the help of Figure \ref{two cuts}, gives the equation.
\end{par}

\begin{equation} \label{intersection of two cuts}
\begin{split}
2 \bigg[ &\big|\delta(A_1)\big| +\big |\delta(A_2)\big|- \Big (\big|\delta(A_1 \cap A_2: \overline{A_1 \cup A_2})\big|+\big|\delta(A_2\setminus A_1: A_1\setminus A_2)\big|\Big) \bigg] 
\\
&= 
\big|\delta(A_1\cap A_2)\big|+\big|\delta(A_2\setminus A_1)\big|+ \big|\delta(A_1\setminus A_2)\big|+\big|\delta(\overline{A_1\cup A_2})\big|.
\end{split}
\end{equation}

\begin{par}
 The propositions about dangerous sets from \cite{NetworkDesign} were later developed and used by Ok, Richter, and Thomassen in \cite{ORT2016linkages} to prove results on linkages in infinite graphs. The following proposition is a special case of Theorem 1.1 in their paper \cite{ORT2016linkages}.
\end{par}

\begin{proposition}\cite{ORT2016linkages}\label{dangerous}
Let $G$ be $(s,k)$-edge-connected and let $s$ be a vertex of $G$ such that $s$ is not incident with a cut-edge and $\deg(s)\neq 3$. Let $F$ be any set of at least two edges, all incident with $s$. Then, no pair of edges in $F$ is $k$-liftable if and only if there is a $k$-dangerous set $A$ so that, for every $sv\in F$, $v\in A$.
\end{proposition}

The core idea in this proposition can be seen more easily in the special case when $|F|=2$. Simply, if a pair is not liftable, this means that lifting it results in a small cut. We formalize this in the following lemma which, unlike Proposition \ref{dangerous}, allows the degree of $s$ to be $3$. Proposition \ref{dangerous} excluded the case $\deg(s)=3$ because in that case it is possible that no pair of edges is liftable, that is $\delta(s)$ is an independent set in $L(G,s,k)$, but then in this case there is no dangerous set containing all the three neighbours of $s$ as we will prove in Lemma \ref{no dangerous set with all three neighbours} below.

\begin{lemma}\label{dangerous set for degree three}
Let $G$ be an $(s,k)$-edge-connected graph such that $\deg(s)\geq 2$. If $su$ and $sv$ are edges in $G$ that are not liftable, then there is a dangerous set $A$ in $G$ containing $u$ and $v$.
\end{lemma}
\begin{proof}
If $su$ and $sv$ are not liftable, then there is a cut $\delta(A)$ in $G_{u,v}$ of size at most $k-1$ such that $s\notin A$ and such that there are two vertices of $G_{u,v}-s$ on different sides of the cut. Both $u$ and $v$ have to be on the side of the cut that does not contain $s$, i.e. $A$, otherwise this gives a cut in $G$ of size at most $k-1$ that is not $\delta(\{s\})$, contradicting that $G$ is $(s,k)$-edge-connected. The side of the cut, $A$, containing $u$ and $v$ but not $s$ is the desired dangerous set.
\end{proof}

\begin{lemma}\label{no dangerous set with all three neighbours}
Let $G$ be an $(s,k)$-edge-connected graph such that $\deg(s)=3$. Any set of vertices in $G$ containing all the neighbours of $s$ is not dangerous.
\end{lemma}
\begin{proof}
If $A$ is a dangerous set, then $|\delta(A)|\leq k+1$. By definition of a dangerous set, $A'=V(G)\setminus (A\cup \{s\})$ is not empty. If all neighbours of $s$ are in $A$, then $|\delta(A')|=|\delta(A\cup \{s\})|\leq (k+1)-3=k-2$, contradicting the fact that $G$ is $(s,k)$-edge-connected.
\end{proof}

\begin{par}
Note that a set $F$ of edges incident with $s$ in $G$ such that no pair of edges in $F$ is $k$-liftable corresponds to an independent set of vertices in the lifting graph $L(G,s,k)$.
 If $G$ is $(s,k)$-edge-connected, and $F$ is an independent set of edges in $L(G,s,k)$, then a dangerous set $A$ as given by Proposition \ref{dangerous} is a dangerous set \textit{corresponding} to $F$.
\end{par}

\begin{par}
Such a dangerous set is not necessarily unique. Consider for example a graph $G$ that consists of a vertex $s$ and three disjoint sets of vertices $A_1$, $A$, and $A_2$, where $s$ has $l\leq \frac{k}{2}$ neighbours in each one of $A_1$ and $A_2$, but no neighbours in $A$, and such that there are $(k+1)-l$ edges between $A_1$ and $A$, and between $A_2$ and $A$, but no edges between $A_1$ and $A_2$. Let $F$ be the set of edges incident with $s$ whose other end-vertices are in $A_1$. Then $A_1$ and $A\cup A_1$ are two different dangerous sets corresponding to $F$.
\end{par}
\begin{par}
However, for every independent set in $L(G,s,k)$, not necessarily maximal, there is a unique minimal corresponding dangerous set in $G$. We prove this in the following lemma. We add the condition $I\neq V(L(G,s,k))$ because when $\deg(s)=3$ it is possible that $V(L(G,s,k))$ is independent, but by Lemma \ref{no dangerous set with all three neighbours} there is no dangerous set corresponding to it.
\end{par}

\begin{lemma} \label{unique minimal dangerous}
Let $G$ be $(s,k)$-edge-connected such that $s$ is not incident with a cut-edge, and let $I$ be an independent set of size at least $2$ in $L(G,s,k)$ such that $I\neq V(L(G,s,k))$. Then $I$ has a unique minimal corresponding dangerous set in $G$.
\end{lemma}
\begin{proof}
\begin{par}
 Proposition \ref{dangerous} guarantees the existence of at least one dangerous set corresponding to $I$ in case $\deg(s)\neq 3$. Lemma \ref{dangerous set for degree three} guarantees the existence of a dangerous set when $\deg(s)=3$ and $I\neq V(L(G,s,k))$. By way of contradiction, suppose that $I$ has two distinct minimal corresponding dangerous sets, $A_1$ and $A_2$, in $G$.
 Note that for $\{i,j\}=\{1,2\}$, 
 $$k+1\geq |\delta(A_i)| = |\delta(\{s\}: A_i)|+|\delta(A_i : A_j\setminus A_i)|+ |\delta(A_i : \overline{A_1\cup A_2 \cup \{s\}})|.$$
 
 \noindent Thus, 
 $$2(k+1)\geq |\delta(A_1)|+|\delta(A_2)|\geq |\delta(A_1\cap A_2)| + |\delta(A_1\cup A_2 : \overline{A_1\cup A_2 \cup \{s\}})|+ |\delta(\{s\}: A_1\cup A_2)|.$$
 
 By minimality, $A_1\cap A_2$ is not dangerous, so $|\delta(A_1\cap A_2)| \geq k+2$. Therefore we have, $|\delta(A_1\cup A_2 : \overline{A_1\cup A_2 \cup \{s\}})|+ |\delta(\{s\}: A_1\cup A_2)|\leq k$. We first show that $\overline{A_1\cup A_2 \cup \{s\}}$ cannot be empty. If it is empty, then $\delta(\{s\}) = \delta(\{s\}: A_1\cup A_2)$, so any edge in $\delta(\{s\}, A_1\cap A_2)$ is an edge that is not liftable with any other edge in $\delta(\{s\})$. Thus, by Frank's Theorem, $|\delta(\{s\}: A_1\cap A_2)|=1$, contradicting our assumption that $|I|\geq 2$.
 \end{par}
 \begin{par}
 Now since $\overline{A_1\cup A_2 \cup \{s\}}$ is not empty, $|\delta(A_1\cup A_2 : \overline{A_1\cup A_2 \cup \{s\}})|+ |\delta(\{s\}: A_1\cup A_2)|\leq k$, and $G$ is $(s,k)$-edge-connected, it follows that 
 $$|\delta(A_1\cup A_2 : \overline{A_1\cup A_2 \cup \{s\}})|+ |\delta(\{s\}: A_1\cup A_2)|= k.$$
 \end{par}
\begin{par}
Let $\ell_1=|\delta(\{s\}: A_1\setminus A_2)|+ |\delta(\overline{A_1\cup A_2 \cup \{s\}}:A_1\setminus A_2)|$, $\ell_2 = |\delta(\{s\}: A_1\cap A_2)|+|\delta(\overline{A_1\cup A_2 \cup \{s\}}: A_1\cap A_2)|$, and $\ell_3= |\delta(\{s\}: A_2\setminus A_1)|+|\delta(\overline{A_1\cup A_2 \cup \{s\}}:A_2\setminus A_1)|$.
\end{par}
 \begin{par}
 Note that $|\delta(A_1\cup A_2 : \overline{A_1\cup A_2 \cup \{s\}})|+ |\delta(\{s\}: A_1\cup A_2)|= \ell_1+ \ell_2 + \ell_3$. Therefore, $\ell_1+ \ell_2 + \ell_3 = k$. Note also that, because $G$ is $(s,k)$-edge-connected, and $A_2$ is dangerous, we have $k\leq \ell_1 + |\delta(A_1\setminus A_2 : A_2)|$, and $\ell_2 + \ell_3 + |\delta(A_1\setminus A_2 : A_2)| \leq k+1$. Thus, $l_2+l_3\leq |\delta(A_1\setminus A_2 : A_2)| \leq \ell_1 +1$. By symmetry, it follows that we also have $\ell_1 +\ell_2 \leq \ell_3+1$. This means that $\ell_2\leq 1$, contradicting the assumption that $|I|\geq 2$. This completes the proof.
 \end{par}\end{proof}

\begin{par}
Now we can define the following.
\end{par}

\begin{definition}\label{the corresponding dangerous} (The Corresponding Dangerous Set)
Let $G$ be an $(s,k)$-edge-connected graph such that $s$ is not incident with a cut-edge, and let $F\subsetneq \delta(\{s\})$ be of size at least $2$ such that no pair of edges in $F$ is liftable. \textit{The dangerous set corresponding to} $F$ is the unique minimal dangerous set $A$ in $G$ containing the non-$s$ end-vertices of the edges in $F$.
\end{definition}

\begin{par}
Most of the claims in this paper work by taking any corresponding dangerous set, not necessarily the minimal one. However, the ones concerning the path and cycle structure need the dangerous sets to be minimal to avoid unnecessary intersections and preserve the structures. This will be reflected in the statements.
\end{par}

The following is Lemma 3.2 in \cite{ORT2016linkages}. It will be needed in our proofs.

\begin{lemma} \cite{ORT2016linkages}\label{intersect}
Let $G$ be an $(s,k)$-edge-connected graph. For $i=1,2$, let $F_i$ be an independent set in $L(G,s,k)$ of size $r_i$ and suppose there is a dangerous set $A_i$ so that $F_i= \delta(\{s\})\cap \delta(A_i)$. Set $\alpha= |F_1\cap F_2|$. If $\alpha > 0$, $r_1> \alpha$, $r_2> \alpha$, and $\overline{A_1\cup A_2 \cup \{s\}}\neq \emptyset$, then $r_1+r_2\leq\lfloor{\deg(s)/2}\rfloor + 2 $.
\end{lemma}

We will also need Lemma 3.3 from \cite{ORT2016linkages}. The first part follows from Franks theorem \ref{Frank}. The second part was proved using dangerous sets.
\begin{lemma}\cite{ORT2016linkages}\label{max}
Let $G$ be an $(s,k)$-edge-connected graph such that $s$ is not incident with a cut-edge.
\begin{enumerate}
    \item [(i)] If $\deg(s)$ is at least $4$, then every independent set in $L(G,s,k)$ has size at most $\lceil{\deg(s)/2}\rceil$ and;
    \item[(ii)] If $\deg(s)$ is even and at least $6$, then any two distinct independent sets in $L(G,s,k)$ of size $\frac{1}{2} \deg(s)$ are disjoint. 
\end{enumerate}
\end{lemma}

\section{Structure of the lifting graph}

\begin{par}
We extend Theorem \ref{ORT} of Ok, Richter, and Thomassen, proving more about the structure of the lifting graph. Theorem \ref{ORT} was proved by induction on the degree of the vertex whose edges are lifted, or in other words the number of vertices of the lifting graph at that vertex. Not all the properties Ok, Richter, and Thomassen showed at the base cases of degrees $4$ and $5$ were carried over in the induction in their proof. We show, without induction, through the study of maximal independent sets that the specific structures they found in the base cases hold in general. We summarize the results of this section in Theorem \ref{summary of lifting graph facts}. In the coming sections we will prove more about its maximal independent sets and its complement. The first proposition below generalizes what was proved for $\deg(5)$ in \cite{ORT2016linkages} to arbitrary odd degree.
\end{par}

\begin{figure}[!h]
     \centering
     \includegraphics[scale=0.8]{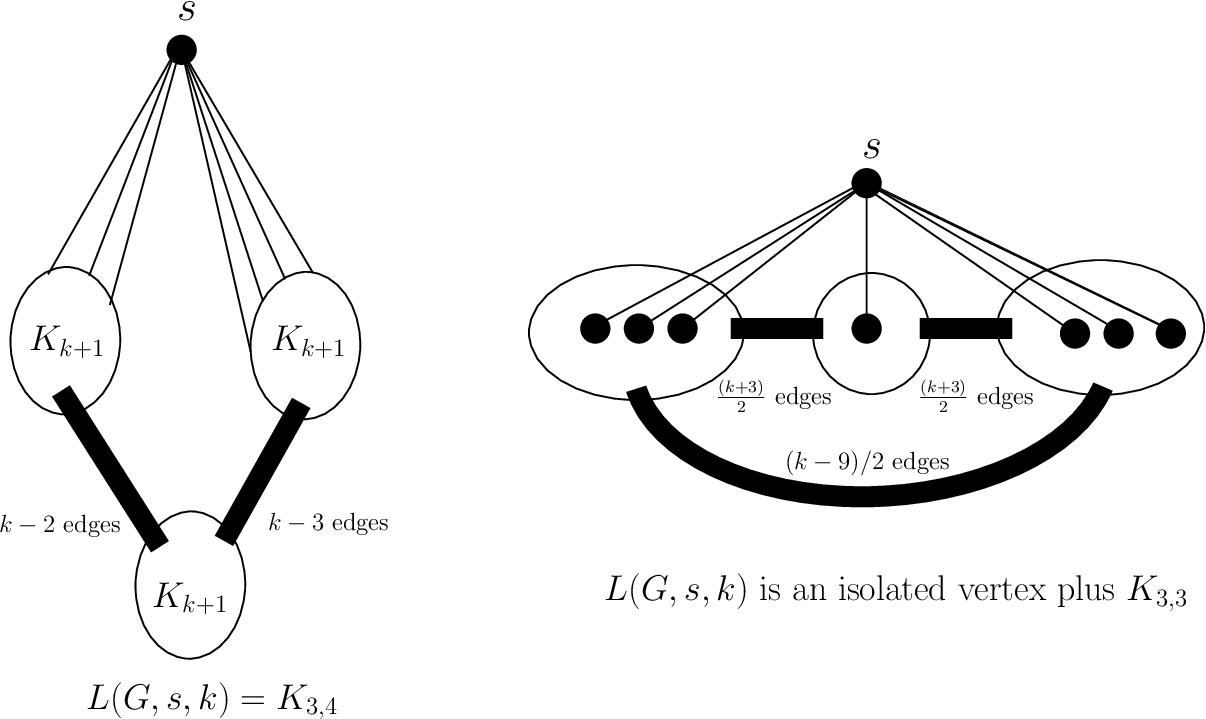}
     \caption{Examples with $\deg(s)=7$ and an independent sets of size $\lceil \deg(s)/2 \rceil$ in $L(G,s,k)$.}
     \label{Bipartite_and_isolated_vertex}
\end{figure}

\begin{proposition}\label{odd}
Let $G$ be an $(s,k)$-edge-connected graph such that $k\geq 2$. Assume that $L(G,s,k)$ contains an independent set of size $\lceil{\deg(s)/2}\rceil$. If $\deg(s)$ is odd and at least $5$, then $L(G,s,k)$ is one of the following (see Figure \ref{Bipartite_and_isolated_vertex}):
                  \begin{enumerate}
                  \item[(i)] a complete bipartite graph with one side of size $\lceil{\deg(s)/2}\rceil$ and the other of size $\lfloor{\deg(s)/2}\rfloor$, or
                  
                  \item[(ii)] an isolated vertex plus a complete bipartite graph with both sides of size $\frac{(\deg(s)-1)}{2}$.
                      
                  \end{enumerate}
\end{proposition}

\begin{proof}
\begin{par}
This is the same as the proof of Case 1 in Proposition 3.5 of \cite{ORT2016linkages}. Let $F$ be a set of edges incident with $s$ of size $\lceil \deg(s)/2 \rceil$ such that no two edges in $F$ form a feasible pair, i.e.~$F$ is independent in $L(G,s,k)$. By Proposition \ref{dangerous} there is a dangerous set $A_1$ containing the non-$s$ ends of the edges in $F$. By definition of a dangerous set, $|\delta_G(A_1)|\leq k+1$.
\end{par}
\begin{par}
Note that the disjoint unions of $\delta(A_1:\overline{A_1\cup \{s\}})$ with $\delta(\{s\}:A_1)$ and with $\delta(\{s\}:\overline{A_1\cup \{s\}})$ respectively give $\delta(A_1)$ and $\delta(\overline{A_1\cup \{s\}})$. By Lemma \ref{max}, $F$ has the maximum possible size of an independent set in $L(G,s,k)$, therefore, $|\delta(\{s\}:A_1)|=|F|=\frac{deg(s)+1}{2}$, $|\delta(\overline{A_1\cup \{s\}})|=|\delta(A_1)|-1\leq k$. Thus $\overline{A_1\cup \{s\}}$ is dangerous, and by Proposition \ref{odd}, $\delta(\{s\})\setminus F$ is an independent set in $L(G,s,k)$. Now we will show that 
\begin{enumerate}
\item[(i)] either every edge in $F$ lifts with every edge in $\delta(\{s\})\setminus F$ (this gives the complete bipartite possibility) or 
\item[(ii)] there is a unique edge in $F$ that does not lift with any edge in $\delta(\{s\})$, but any other edge of $F$ lifts with each edge of $\delta(\{s\})\setminus F$. This gives the isolated vertex plus complete bipartite case.
\end{enumerate}
\end{par}
\begin{par}
In particular we will show that if an edge $e_1\in F$ does not lift with an edge $e_2\in \delta(\{s\})\setminus F$, then $e_1$ does not lift with any other edge, and then by Frank's Theorem \ref{Frank} it is the only such edge. We present this in the following claim.
\end{par}
\begin{claim} \label{claim odd}
Let $e_1\in F$ and $e_2\in \delta(\{s\})\setminus F$ be a pair of edges that is not $k$-liftable. Then $e_1$ is not $k$-liftable with any edge in $\delta(\{s\})\setminus \{e_1\}$. 
\end{claim}
\begin{proof}
\begin{par}
Suppose $e_1\in F$ and $e_2 \in \delta(\{s\})\setminus F$ do not form a $k$-liftable pair and let $A_2$ be a dangerous set that witnesses this (such a set exists by Lemma \ref{dangerous set for degree three}), so the non-$s$ ends of $e_1$ and $e_2$ are in $A_2$. By Lemma \ref{max} the maximum size of an independent set in $L(G,s,k)$ is $\lceil \deg(s)/2 \rceil$; consequently, at most $\lceil \deg(s)/2 \rceil$ of the edges incident with $s$ have their non-$s$ ends in $A_2$. Therefore, at least $\lfloor{\deg(s)/2}\rfloor$ of the edges incident with $s$ have their non-$s$ ends in $\overline{A_2\cup \{s\}}$. Now since $e_2$ is in $\delta(\{s\})\setminus F$ but has its non-$s$ end in $A_2$ and $|\delta(\{s\})\setminus F|=\lfloor{\deg(s)/2}\rfloor$, the set of edges incident with $s$ whose non-$s$ ends are in $\overline{A_2\cup \{s\}}$ contains an edge $e$ from $F$. Also since $e_1\in F$ has its non-$s$ end in $A_2$, $e$ is in $F\setminus \{e_1\}$.
\end{par}
\begin{par}
  For $i=1,2$, set $F_i=\delta(A_i)\cap \delta(s)$; in particular $F_1=F$. Three of the hypotheses of Lemma \ref{intersect} are satisfied: $e_1\in F_1\cap F_2$ ($\alpha>0$), $e_2\in F_2\setminus F_1$ ($|F_2|>\alpha$), and $e\in F_1\setminus F_2$ ($|F_1|>\alpha$). If the other hypothesis of the lemma $\overline{A_1\cup A_2 \cup \{s\}}\neq \emptyset$ is also satisfied, then $|F_1|+|F_2|\leq \lfloor \deg(s)/2 \rfloor +2$, i.e.
$\lceil \deg(s)/2 \rceil+ |F_2|\leq \lfloor{\deg(s)/2}\rfloor + 2 $. Since $\deg(s)$ is odd, this means that $|F_2|\leq 1$, a contradiction since $F_2$ contains both $e_1$ and $e_2$. Thus, $\overline{A_1\cup A_2 \cup \{s\}}= \emptyset$, so all the edges of $\delta(\{s\})\setminus F$ have their non-$s$ ends in $A_2\setminus A_1$. 
\end{par}
\begin{par}
Now since, $|\delta(\{s\})\setminus F|= \lfloor{\deg(s)/2}\rfloor$ and there is an edge from $s$ to $A_1\cap A_2$, we have $|\delta(A_2)\cap \delta(\{s\})|\geq \lceil \deg(s)/2 \rceil$. By Lemma \ref{max} this is the maximum possible size of an independent set in $L(G,s,k)$, and since $A_2$ is dangerous, $|\delta(A_2)\cap \delta(\{s\})|= \lceil \deg(s)/2 \rceil$. 
\end{par}
\begin{par}
Consequently, $|\delta(\{s\} : (A_1 \cap A_2))|=1$, and all the non-$s$ ends of the edges of $F$ other than $e_1$ are in $A_1\setminus A_2$. Since all the non-$s$ ends of the edges incident with $s$ other than $e_1$ are either in $A_1\setminus A_2$ or $A_2\setminus A_1$, and $e_1$ has its non-$s$ end in $A_1\cap A_2$, $e_1$ does not lift with any other edge.
\end{par}\end{proof}

\begin{par}
By Frank's Theorem \ref{Frank}, there can only be one edge $e$ in $\delta(\{s\})$ that is not liftable with any other edge in $\delta(\{s\})$. If such an edge $e$ exists, and  if $f$ is an edge in $F\setminus\{e\}$ and $f'$ is an edge in $\delta(\{s\})\setminus F$, then $f$ and $f'$ is a $k$-liftable pair, otherwise, by Claim \ref{claim odd}, $f$ is not liftable with any edge in $\delta(\{s\})\setminus \{f\}$, contradicting the uniqueness of $e$. Thus every edge in $F\setminus \{e\}$ lifts with every edge in $\delta(\{s\})\setminus F$.
\end{par}
\begin{par}
 This gives the structure of an isolated vertex ($e$) plus a balanced complete bipartite graph for the lifting graph. 
 \end{par}\end{proof}

\begin{par}
This proposition dealt with the case when $\deg(s)$ is odd and $L(G,s,k)$ contains an independent set of the maximum possible size $\lceil \deg(s)/2 \rceil$, Figure \ref{Bipartite_and_isolated_vertex}. There it is proved that the maximal independent sets are either disjoint or intersect in exactly one vertex. Point $(4)$ in Theorem \ref{ORT} and \cite[Theorem~3.2]{Jordan1999ConstrainedSplitting} showed that if both $\deg(s)$ and $k$ are even, then the maximal independent sets of $L(G,s,k)$ are disjoint (complete multipartite). The case of an isolated vertex plus complete bipartite means two maximal independent sets of $L(G,s,k)$ intersecting in exactly one vertex. Now we consider the other cases.
\end{par}

\begin{par}
The new idea here is that we focus on the maximal independent sets of the lifting graph. In the following lemmas and theorem we will only see how they intersect. After that, in the coming sections, we will use this to find out what kind of structure they are arranged into, a path, a cycle, pairwise disjoint, or otherwise.
\end{par}

\begin{figure}[!h]
\centering
   \includegraphics[scale=0.6]{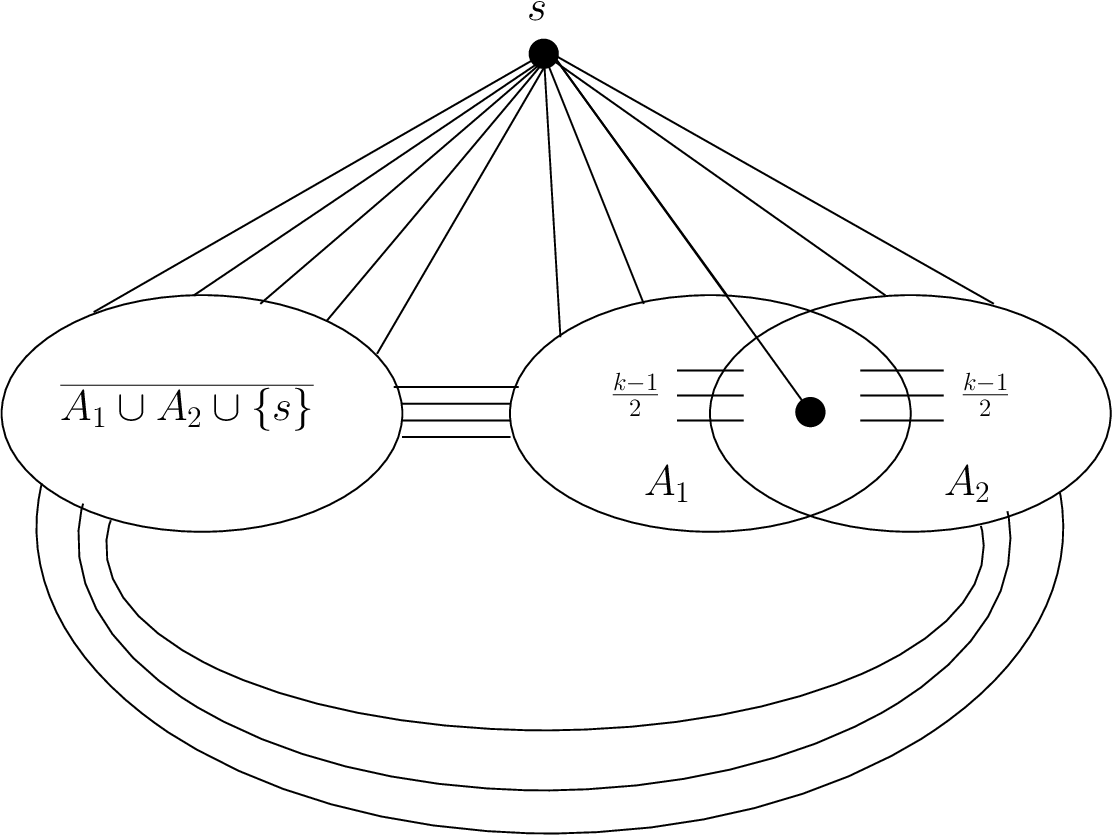} 
   \caption{Two intersecting dangerous sets corresponding to maximal independent sets not both of size $\lceil \deg(s)/2 \rceil$.}
   \label{two_intersecting_dangerous_sets_of_small_size}
\end{figure}

The following lemma is the first here on maximal independent sets not necessarily of size $\lceil{\deg(s)/2}\rceil$. Note that the implication below that $k$ is odd if the two sets intersect and their union is not the entire vertex set (i.e. they are crossing sets) also follows from point (c) in \cite[Lemma~2.2]{Jordan1999ConstrainedSplitting} even though there it was proved, as the other results in that paper, under the standing assumption at the beginning of each section that $\deg(s)$ is even. However, the proof there of that point in particular did not make use of that assumption, and it did not compare maximal independent sets to dangerous sets, but was a result about maximal dangerous sets.

\begin{lemma}\label{two dangerous}
Let $G$ be an $(s,k)$-edge-connected graph such that $k\geq 2$ and $\deg(s)\geq 4$. Suppose that $I_1$ and $I_2$ are two maximal independent sets in $L(G,s,k)$ of size at least $2$ each, and let $A_1$ and $A_2$ be two dangerous sets in $G$ corresponding to $I_1$ and $I_2$ respectively.

\begin{enumerate}

\item[(i)] Then $|I_1\cap I_2|\leq 1$.

\item [(ii)] If $|I_1\cap I_2|= 1$ and $I_1\cup I_2 \neq V(L(G,s,k))$, then $k$ is odd, and

\begin{enumerate}
\item [(a)] $|\delta_G(A_1)|= |\delta_G(A_2)|= k+1$;
\item[(b)] $|\delta_{G}(A_2\setminus A_1: A_1\setminus A_2)|=0$;
\item [(c)] $|\delta_{G-s}(A_1\cap A_2: \overline{A_1\cup A_2})|=0$; and
\item[(d)]$|\delta_{G}(A_2\setminus A_1: A_1\cap A_2)|= |\delta_{G}(A_1\cap A_2: A_1\setminus A_2)| = \frac{k-1}{2}$.
\end{enumerate}

\end{enumerate}

\end{lemma}

\begin{proof}
\begin{par}
We follow a generalized version of the proof of Case 2 in Proposition 3.5 of \cite{ORT}. Let $k_1= |\delta(\{s\}: A_1\setminus A_2)|$, $k_2=|\delta(\{s\}: A_2\setminus A_1)|$, $k_3=|\delta(\{s\}: A_1\cap A_2)|$, and assume that $k_3\neq 0$.
\end{par}
\begin{par}
Since $A_1$ and $A_2$ are dangerous, we have $|\delta_G(A_1)|\leq k+1$ and $|\delta_G(A_2)|\leq k+1$. In particular,
$|\delta_{G-s}(A_1)|\leq (k+1)-(k_1+k_3)$, and
$|\delta_{G-s}(A_2)|\leq (k+1)-(k_2+k_3)$.
\end{par}
\begin{par}
 Since $I_1$ and $I_2$ are maximal independent sets in $L(G,s,k)$, $I_1\cup I_2$ is not independent, consequently, by Proposition \ref{dangerous}, $A_1\cup A_2$ is not dangerous.
 \end{par}
 \begin{par}
     Also because $I_i$ is a maximal independent set of $L(G,s,k)$ for $i\in \{1,2\}$, each $A_i$ does not contain neighbours of $s$ other than those that are end-vertices of edges in $I_i$.
 \end{par}
 \begin{par}
  This and the assumption that $I_1\cup I_2 \neq V(L(G,s,k))$ imply that at least one edge incident with $s$ has its non-$s$ end in the set $\overline{A_1\cup A_2\cup \{s\}}$. Thus, $A_1\cup A_2$ is separating in $G-s$ in the sense that $V(G-s)\setminus (A_1\cup A_2)$ and $A_1\cup A_2$ are both non-empty. By the definition of a dangerous set, the only way for $A_1\cup A_2$ to not be dangerous is if $|\delta_G(A_1\cup A_2)|\geq k+2$. 
  This means that $|\delta_{G-s}(A_1\cup A_2)|\geq (k+2)-(k_1+k_2+k_3)$. Note that in $G-s$, $s$ is not in $\overline{A_1\cup A_2}$.
\end{par}
\begin{par}
Also since $G$ is $(s,k)$-edge-connected, $|\delta_{G-s}(A_1\cap A_2)|\geq k-k_3$, $|\delta_{G-s}(A_1 \setminus A_2)|\geq k-k_1$, and $|\delta_{G-s}(A_2 \setminus A_1)|\geq k-k_2$. Observe that in $G-s$:
\end{par}

\begin{align*}
\begin{split}
& 2\bigg[\big|\delta(A_1)\big| +\big|\delta(A_2)\big|- \Big(\big|\delta (A_1 \cap A_2: \overline{A_1 \cup A_2})|+\big|\delta (A_2\setminus A_1: A_1\setminus A_2)\big|\Big)\bigg]
\\
&\leq 2\bigg [(k+1)-(k_1+k_3)+ (k+1)-(k_2+k_3) \bigg]
\\
& =4k-2(k_1+k_2+k_3)+2+(2-2k_3)
\end{split}
\end{align*}

and  
\begin{align*}
\begin{split}
&\big|\delta (A_1\cap A_2)\big|+\big|\delta (A_2\setminus A_1)\big|+ \big |\delta (A_1\setminus A_2)\big|+ \big|\delta (\overline{A_1\cup A_2})\big| \geq \\ & (k-k_3)+(k-k_2)+(k-k_1)+(k+2)-(k_1+k_2+k_3)\\
& =4k-2(k_1+k_2+k_3)+2.
\end{split}
\end{align*}
\begin{par}
By Equation \ref{intersection of two cuts}, it follows that $2-2k_3\geq 0$, i.e. $k_3\leq 1$ as desired.
\end{par}

 If $k_3=1$, then inequalities throughout have to be equalities. More precisely,
\begin{enumerate}
\item[(1)] $|\delta_{G-s}(A_1)|=|\delta_{G-s}(A_1 \setminus   A_2)|= k-k_1$, so $|\delta_{G}(A_1\setminus A_2)|=k$ and $|\delta_G(A_1)|=k+1$;
\item[(2)] $|\delta_{G-s}(A_2)|= |\delta_{G-s}(A_2 \setminus A_1)|=k-k_2$, so $|\delta_{G}(A_2\setminus A_1)|=k$ and $|\delta_G(A_2)|=k+1$;
\item[(3)] $|\delta_{G-s}(A_1\cap A_2)|= k-1$;
\item[(4)] $|\delta_{G-s}(\overline{A_1\cup A_2})|=(k+2)-(k_1+k_2+k_3)=k+1-k_1-k_2$; 
\item[(5)] $|\delta_{G-s}(A_1\cap A_2: \overline{A_1\cup A_2})|=|\delta_{G-s}(A_2\setminus A_1: A_1\setminus A_2)|=0$.
\end{enumerate}
\begin{par}
  \noindent Together, $(1)$ and $(2)$ give $(a)$, and $(5)$ gives $(b)$ and $(c)$. To prove $(d)$ note that $(5)$ also gives (cf. Figure \ref{two cuts}),
\newline
$|\delta_{G-s}(A_2)|=|\delta_{G-s}(A_2\setminus A_1: \overline{A_1\cup A_2})|+|\delta_{G-s}(A_1\cap A_2: A_1\setminus A_2)|$, and \newline
$|\delta_{G-s}(A_2\setminus A_1)|=|\delta_{G-s}(A_2\setminus A_1: \overline{A_1\cup A_2})| +|\delta_{G-s}(A_2\setminus A_1: A_1\cap A_2)|$. From $(2)$ we have $|\delta_{G-s}(A_2)|= |\delta_{G-s}(A_2 \setminus A_1)|$. Cancelling the common $|\delta_{G-s}(A_2\setminus A_1: \overline{A_1\cup A_2})|$ on both sides yields $$|\delta_{G-s}(A_2\setminus A_1: A_1\cap A_2)|=|\delta_{G-s}(A_1\cap A_2: A_1\setminus A_2)|.$$

Now this last equality, $(3)$, and $(5)$ imply that, \newline
$k-1=|\delta_{G-s}(A_1\cap A_2)|= |\delta_{G-s}(A_1\cap A_2, A_1\setminus A_2)|+ |\delta_{G-s}(A_1\cap A_2, A_2\setminus A_1)|$
\newline
=$2|\delta_{G-s}(A_1\cap A_2, A_2\setminus A_1)|$. Thus $k$ has to be odd and, 
$$
|\delta_{G-s}(A_1\cap A_2, A_1\setminus A_2)| = |\delta_{G-s}(A_2\setminus A_1, A_1\cap A_2)|= \frac{k-1}{2}. 
$$
\end{par}

\begin{par}
Since $\deg(s)>3$, the maximum size of an independent set is $\lceil \deg(s)/2 \rceil$ by Frank's theorem \ref{Frank}. Therefore if the union of two intersecting maximal independent sets is the entire vertex set of $L(G,s,k)$, then they are both of size $\lceil \deg(s)/2 \rceil$ and they intersect in exactly one vertex.\end{par}\end{proof}

\begin{figure}[!h]
     \centering
     \includegraphics[scale=0.7]{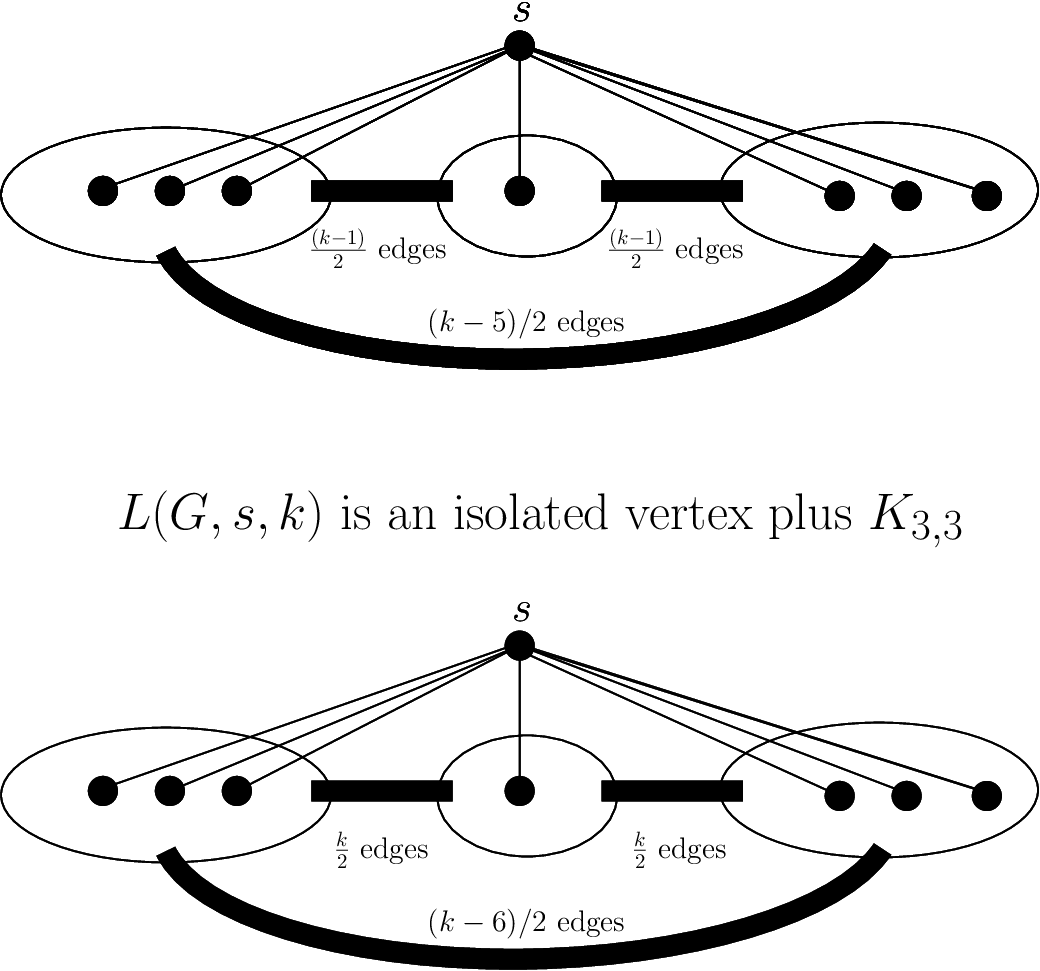}
     \caption{Examples where the lifting graph is an isolated vertex plus a complete balanced bipartite graph for even and odd $k$.}
     \label{examples_k_even_and_odd_vertex_plus_bipartite}
\end{figure}

\begin{par}
From this Lemma we have as a corollary that when two maximal independent sets in the lifting graph - not both of the large size $\lceil \deg(s)/2 \rceil$ - intersect, then $k$ has to be odd. But there is no restriction on the parity of $\deg(s)$ in that case. On the other hand, if two independent sets of size $\lceil \deg(s)/2 \rceil$ intersect, then by Lemma \ref{max} either $\deg(s)$ is odd or is equal to $4$. In case $\deg(s)$ is odd, then $L(G,s,k)$ is an isolated vertex plus a balanced complete bipartite graph as shown in Lemma \ref{odd}. The parity of $k$ in that case can be even as well as odd (examples in Figure \ref{examples_k_even_and_odd_vertex_plus_bipartite} and the right drawing of Figure \ref{Bipartite_and_isolated_vertex}). In case $\deg(s)=4$, such an intersection can only happen if $k$ is odd as shown in the comment after Proposition 3.4 in \cite{ORT2016linkages}, and it was also shown there that the graph $G$ has a specific structure in that case, to be presented later.
\end{par}
\begin{par}
  Another thing that was proved in the previous lemma is that the dangerous sets $A_1$ and $A_2$ in $G$ corresponding to the two intersecting maximal independent sets of $L(G,s,k)$ have a fixed number of edges, $(k-1)/2$, between $A_1\cap A_2$ and $A_1\setminus A_2$ as well as between $A_1\cap A_2$ and $A_2\setminus A_1$, a number depending only on the connectivity. By the example illustrated in the right drawing in Figure \ref{Bipartite_and_isolated_vertex} this does not have to be the case when the two intersecting maximal independent sets are both of size $\lceil \deg(s)/2 \rceil$. However, in the following lemma we will see that the numbers of edges between $A_1\cap A_2$ and each of the sets $A_1\setminus A_2$ and $A_2\setminus A_1$ still has to be equal, and its sum with the number of edges between $A_1\setminus A_2$ and $A_2\setminus A_1$ is a constant of the graph depending only on the connectivity and the degree of $s$. Recall that when $\deg(s)=3$ it is possible that the edges incident with $s$ form an independent set of size $3$ in $L(G,s,k)$. In this case, sets of size $\lceil \deg(s)/2 \rceil=2$ are not maximal independent. This is why we use `independent' instead of `maximal independent' in the following lemma.
\end{par}

\begin{lemma}\label{large independent}
 Let $G$ be an $(s,k)$-edge-connected graph such that $k\geq 2$. Suppose that $\deg(s)\geq 3$, $\deg(s)\neq 4$, and two independent sets of size $\lceil \deg(s)/2 \rceil$ in $L(G,s,k)$ have a non-empty intersection. Then $\deg(s)$ is odd, and if $A_1$ and $A_2$ are dangerous sets in $G$ corresponding to those two independent sets in $L(G,s,k)$, then 
\begin{enumerate}
\item[(1)]there are no vertices outside $A_1\cup A_2\cup\{s\}$;
\item[(2)]$|\delta_{G}(A_1\setminus A_2: A_2\setminus A_1)| \leq (k-\deg(s)+2)/2$;
\item[(3)]$|\delta_{G}(A_2\setminus A_1: A_1\cap A_2)|= |\delta_{G}(A_1\cap A_2: A_1\setminus A_2)| \geq \frac{k-1}{2}$;
\item[(4)] $|\delta_G(A_1)|=|\delta_G(A_2)|=k+1$;

\item[(5)] $|\delta_{G}(A_1\setminus A_2: A_2\setminus A_1)|+|\delta_{G}(A_1\cap A_2: A_1\setminus A_2)|=(k+1)-\frac{\deg(s)+1}{2}$; 
and
\item[(6)] $\deg(s)\leq k+2$.

\end{enumerate}
\end{lemma}

\begin{proof}

\begin{par}
Suppose that $L(G,s,k)$ contains two independent sets of the size $\lceil \deg(s)/2 \rceil$ with non-empty intersection, and that $deg(s) \neq 4$. By Lemma \ref{max} $\deg(s)$ is odd. By Proposition \ref{odd}, if $\deg(s)\neq 3$, then the two sets intersect in exactly one vertex and they are the only maximal independent sets of $L(G,s,k)$ (this is the isolated vertex plus complete bipartite case). In case $\deg(s)=3$ an independent set of size $2$ does not have to be maximal as it is possible that all three edges incident with $s$ form an independent set of size $3$. Figure \ref{deg(s)=3} in the next section provides examples where all three edges at $s$ form an independent set.
\end{par}
\begin{par}
  By Lemma \ref{dangerous}, in case $\deg(s)>4$, there are two dangerous sets $A_1$ and $A_2$ in $G$ corresponding to the two maximal independent sets. For $\deg(s)=3$ the existence of such dangerous sets is guaranteed by Lemma \ref{dangerous set for degree three}, and by Lemma \ref{no dangerous set with all three neighbours} each one of those dangerous sets contains exactly two of the neighbours of $s$ (the third neighbour is outside it but in the other dangerous set). In any case, the union $A_1\cup A_2$ is not dangerous, and hence  $|\delta_{G}(\overline{A_1\cup A_2})|\geq k+2$ if $\overline{A_1\cup A_2\cup \{s\}}\neq \emptyset$.
\end{par}
\begin{par}
  We again use, in $G-s$, the same standard equation we used before \ref{intersection of two cuts}. We first show that there are no vertices in $G-s$ outside the union of two dangerous sets corresponding to the two maximal independent sets of $L(G,s,k)$. 
 \end{par} 
  
  \begin{claim}
  $A_1\cup A_2= V(G)\setminus \{s\}$.
  \end{claim}
  \begin{proof}
  \begin{par}
  Suppose by way of contradiction that $\overline{A_1\cup A_2\cup \{s\}}\neq \emptyset$. Then the right hand side of equation \ref{intersection of two cuts} applied in $G-s$ is at least $$(k-1)+\Big(k-\Big(\frac{\deg(s)-1}{2}\Big)\Big)+\Big(k-\Big(\frac{\deg(s)-1}{2}\Big)\Big)+((k+2)-\deg(s)).$$
    Note that $s$ has exactly one neighbour in $A_1\cap A_2$, and this follows from the last paragraph in the statement of Lemma \ref{two dangerous}. This, and the $k$-edge-connectivity of $G$, give the first three terms. The last term follows from the fact that $A_1\cup A_2$ is not dangerous.
  \end{par}
  \begin{par}
    This is a lower bound of $4k-2\deg(s)+2$. On the other hand, the left side has the upper bound of $2[(k+1-(\frac{\deg(s)+1}{2}))+(k+1-(\frac{\deg(s)+1}{2}))]= 4k+4-2(\deg(s)+1)=4k-2\deg(s)+2$. Thus, both sides are equal to $4k-2\deg(s)+2$, and the individual upper and lower bounds on each term hold with equality. In particular $|\delta_{G-s}(\overline{A_1\cup A_2})| = ((k+2)-\deg(s))$.
  \end{par}
  \begin{par}
    The set $\overline{A_1\cup A_2}$ does not contain any neighbours of $s$, as $A_1$ and $A_2$ contain all the neighbours of $s$. Therefore, $|\delta_{G}(\overline{A_1\cup A_2})| = |\delta_{G-s}(\overline{A_1\cup A_2})| \newline  = ((k+2)-\deg(s)) <k$, a contradiction.
  \end{par}
  \end{proof}
  
  \begin{par}
    Now, knowing that $\overline{A_1\cup A_2\cup \{s\}}= \emptyset$, then the lower bound on the right side of equation \ref{intersection of two cuts} is $(k-1)+(k-(\frac{\deg(s)-1}{2}))+(k-(\frac{\deg(s)-1}{2}))=3k-\deg(s)$. 
  \end{par}
\begin{par}  
    The upper bound on the left side is $4k-2\deg(s)+2=(3k-\deg(s))+(k-\deg(s)+2)$. This means that, cf. equation \ref{intersection of two cuts}, \newline $(|\delta_{G-s}(A_1\cap A_2:\overline{A_1\cup A_2})|+|\delta_{G-s}(A_2\setminus A_1: A_1\setminus A_2)|) \leq (k-deg(s)+2)/2$.
    \end{par}
    \begin{par}
  We know that $|\delta_{G-s}(A_1\cap A_2:\overline{A_1\cup A_2})|=0$ as $\overline{A_1\cup A_2\cup \{s\}}=\emptyset$. Thus, $|\delta_{G-s}(A_2\setminus A_1: A_1\setminus A_2)| \leq (k-\deg(s)+2)/2$, and this same upper bound also holds in $G$.
\end{par}
\begin{par}
 Since $G$ is $k$-edge-connected and, for $i=1,2$, $A_i$ is dangerous,$|\delta(A_i)|$ is either $k$ or $k+1$. 
 \end{par}
 \begin{claim}
 For $i=1,2$, $|\delta(A_i)|=k+1$.
 \end{claim}
 \begin{proof}
 \begin{par}
 If, say, $|\delta_G(A_1)|=k$, then,\newline
  $|\delta_G(A_1\setminus A_2 : A_2\setminus A_1)|+|\delta_G(A_1\cap A_2: A_2\setminus A_1)| + |\delta(\{s\}:A_1)|=k$. It follows that $|\delta_G(A_2\setminus A_1)|= k-1$, as $s$ has exactly one neighbour in $A_1\cap A_2$ (Lemma \ref{two dangerous}) and $|\delta_G(\{s\}:A_2\setminus A_1)|=|\delta_G(\{s\}:A_1\setminus A_2)|$ (as $A_1$ and $A_2$ correspond to maximal independent sets of the same size), a contradiction. The same argument holds for $A_2$.
\end{par}
\end{proof}

\begin{par}
\noindent The equality $|\delta_G(A_1)|=|\delta_G(A_2)|=k+1$ means that $|\delta_{G-s}(A_1)|=|\delta_{G-s}(A_2)|=k+1-(\frac{\deg(s)+1}{2})$, i.e. \newline $|\delta_G(A_1\setminus A_2: A_2\setminus A_1)|+|\delta_G(A_1\cap A_2: A_2\setminus A_1)| = \newline |\delta_G(A_2\setminus A_1: A_1\setminus A_2)|+|\delta_G(A_1\cap A_2: A_1\setminus A_2)|= k+1-(\frac{\deg(s)+1}{2})$. 
\end{par}

\noindent In particular, $$|\delta_G(A_1\cap A_2, A_2\setminus A_1)|=|\delta_G(A_1\cap A_2, A_1\setminus A_2)|.$$

\begin{par}
  Now since $|\delta_G(A_1\cap A_2)|\geq k$ and $s$ has exactly one neighbour in $A_1\cap A_2$, both \newline $|\delta_G(A_1\cap A_2: A_2\setminus A_1)|$ and $|\delta_G(A_1\cap A_2: A_1\setminus A_2)|$ have to be at least $(k-1)/2$.
\end{par}

\begin{par}
  The lower bound $|\delta_{G}(A_2\setminus A_1: A_1\cap A_2)|= |\delta_{G}(A_1\cap A_2: A_1\setminus A_2)| \geq \frac{k-1}{2}$, and the fact that $A_1$ and $A_2$ are both dangerous and each contain $(\deg(s)+1)/2$ neighbours of $s$, imply that $(\deg(s)+1)/2\leq (k+3)/2$. Thus, $\deg(s)\leq k+2$.
\end{par} \end{proof}

\begin{remark}
\begin{par}
  It is possible that $k$ is even in the previous lemma. That is, it is possible to have the lifting graph of isolated vertex plus complete bipartite with even $k$ (only $\deg(s)$ has to be odd). In that case $|\delta_G(A_1\cap A_2: A_2\setminus A_1)|$ and $|\delta_G(A_1\cap A_2: A_1\setminus A_2)|$ will be at least $k/2$. See Figure \ref{examples_k_even_and_odd_vertex_plus_bipartite} and the right drawing of Figure \ref{Bipartite_and_isolated_vertex}.
\end{par}
\end{remark}

Now we summarize what we know so far on the lifting graph and maximal independent sets.

\begin{theorem}\label{summary of lifting graph facts}
Let $G$ be an $(s,k)$-edge-connected graph such that $k\geq 2$ and $\deg(s)\geq 4$.

\begin{enumerate}
\item[(1)]  Any two maximal independent sets of $L(G,s,k)$ intersect in at most one vertex.

\item[(2)] If the union of two intersecting maximal independent sets in $L(G,s,k)$ equals $V(L(G,s,k))$, then $\deg(s)$ is odd and both sets are of size $(\deg(s)+1)/2$.

\item[(3)] If $\deg(s) >4$ and there exists two intersecting maximal independent sets of $L(G,s,k)$ of size $\lceil \deg(s)/2\rceil$, then $L(G,s,k)$ consists of an isolated vertex and a balanced complete bipartite graph. In particular, 

\begin{enumerate}
\item [(i)] these are the only two maximal independent sets of $L(G,s,k)$; and
\item [(ii)] the isolated vertex is their intersection.
\end{enumerate}

\item[(4)] The maximal independent sets of $L(G,s,k)$ are pairwise disjoint ($L(G,s,k)$ is complete multipartite) if $k$ is even and one of the following holds:
\begin{enumerate}
\item [(i)] $\deg(s)$ is even, or

\item[(ii)] at most one independent set of $L(G,s,k)$ has size $\lceil \deg(s)/2 \rceil$.
\end{enumerate}

\end{enumerate}

Moreover, if $L(G,s,k)$ is complete multipartite, then it is not a star.
\end{theorem}

\begin{proof}

\begin{par}
From Lemma \ref{two dangerous}, we have $(1)$. Suppose that there exists two intersecting maximal independent sets of $L(G,s,k)$ whose union is $V(L(G,s,k))$. By Frank's theorem \ref{Frank}, every independent set in $L(G,s,k)$ is of size at most $\lceil \deg(s)/2 \rceil$. Thus the only way the union of two intersecting independent sets can equal $V(L(G,s,k))$ is if $\deg(s)$ is odd and they are both of size $(\deg(s)+1)/2$. This proves $(2)$.
\end{par}
\begin{par}
Now suppose that there exists two intersecting maximal independent sets of $L(G,s,k)$ of size $\lceil \deg(s)/2\rceil$. If $\deg(s)>4$, then by Lemma \ref{large independent} it follows that $\deg(s)$ is odd, and then by Proposition \ref{odd} we have that $L(G,s,k)$ consists of an isolated vertex and a balanced complete bipartite graph. This completes the proof of $(3)$.
\end{par}
\begin{par}
If $k$ is even and $\deg(s)$ is even, then the maximal independent sets of $L(G,s,k)$ are pairwise disjoint by $(4)$ in Theorem \ref{ORT}. If $\deg(s)$ is odd and at most one independent set of $L(G,s,k)$ has size $\lceil \deg(s)/2 \rceil$, then either $L(G,s,k)$ consists of two disjoint maximal independent sets of sizes $\lceil \deg(s)/2 \rceil$ and $\lfloor \deg(s)/2 \rfloor$, or for any two maximal independent sets $I_1$ and $I_2$, $I_1\cup I_2 \neq V(L(G,s,k))$, and by Lemma \ref{two dangerous} it follows that in case $k$ is even, $I_1\cap I_2 = \emptyset$. This proves $(4)$.
\end{par}

\begin{par}
  Since $\deg(s)\geq 4$, then by the theorem of Frank \ref{Frank}, $L(G,s,k)$ cannot be a star because a star on $n$ vertices contains a maximal independent set of size $n-1 > \lceil \frac{n}{2}\rceil$.
\end{par} \end{proof}

\begin{remark}
 Every complete multipartite graph $K_{k_1,\cdots, k_m}$ satisfying the following conditions is the $k$-lifting graph of some graph: $m\geq 2$ and $k_i \leq \min(k+1,\lceil \deg(s)/2 \rceil)$ for every $i\in \{1,\cdots, m\}$, and for every proper subset $S\subsetneq \{1,\cdots, m\}$, $\sum_{i\in S} k_i + \sum_{i \notin S} (k+1-k_i)$ is at least $k+2$ if $|S|\geq 2$, and at least $k$ if $S$ is a singleton or empty. If $m=2$, then there is no proper subset of $\{1,\cdots, m\}$ of size at least $2$. This means that in that case only the lower bound of $k$, and not $k+2$, is assumed, as in the graph illustrated in the left drawing of Figure \ref{Bipartite_and_isolated_vertex}, which we are going to generalize now to construct a graph $G$ such that $L(G,s,k)= K_{k_1,\cdots, k_m}$. Let $A, A_1,\cdots, A_m$ be cliques of size $k+1$ each. Let $s$ have $k_i$ edges to $A_i$, and let $A$ have $k+1-k_i$ edges to $A_i$. The resulting graph $G$ is such that $G$ is $(s,k)$-edge-connected and lifting any pair of edges with end-vertices in different $A_i$'s results in an $(s,k)$-edge-connected graph (any cut that does not have $s$ alone on one side is of size at least $k$). Each $A_i$ is a dangerous set (since $|\delta(A_i)|=k+1$), thus for each $i$, $\delta(\{s\}:A_i)$ is a maximal independent set in $L(G,s,k)$ of size $k_i$.
\end{remark}

\begin{par}
  To be able to talk more neatly about a connected collection of maximal independent sets in $L(G,s,k)$, we define the \textit{independence graph}.
\end{par}

\begin{definition} (Independence Graph)
For a graph $H$ the independence graph $I(H)$ is the graph whose vertex set is the set of maximal independent sets of $H$ and in which two vertices are adjacent if and only if the corresponding independent sets have a nonempty intersection.
\end{definition}
\begin{remark}
   The complement of $H$ is connected if and only if $I(H)$ is connected.
\end{remark}

\begin{par}
Rereading Theorem \ref{summary of lifting graph facts} we see that when $L(G,s,k)$ is an isolated vertex plus a balanced complete bipartite graph, $I(L(G,s,k))$ is a path of length one (one edge representing the intersection of two maximal independent sets of size $(\deg(s)+1)/2$), and when $L(G,s,k)$ is complete multipartite graph, $I(L(G,s,k))$ consists of singletons.
\end{par}

\begin{par}
  The maximal independent sets form, make, or are arranged in, a path or a cycle if the independence graph, which is the intersection graph of the maximal independent sets, is respectively a path or a cycle. However, we will also need to talk about similar arrangements for the corresponding dangerous sets. We therefore introduce the following general definition.
\end{par}

\begin{definition}
A collection of sets $S_1, \cdots, S_n$ \textit{makes a path or a cycle via intersections} if the graph with vertex set $\{S_1, \cdots, S_n\}$, in which two vertices are adjacent iff they have a nonempty intersection, is a path or a cycle respectively.
\end{definition}

\begin{par}
 In Section 3.2 of \cite{ORT2016linkages} the following was proved (see Figure \ref{degree_4_cycle}). 
\end{par}

\begin{figure}[!h]
     \centering
     \includegraphics[scale=0.7]{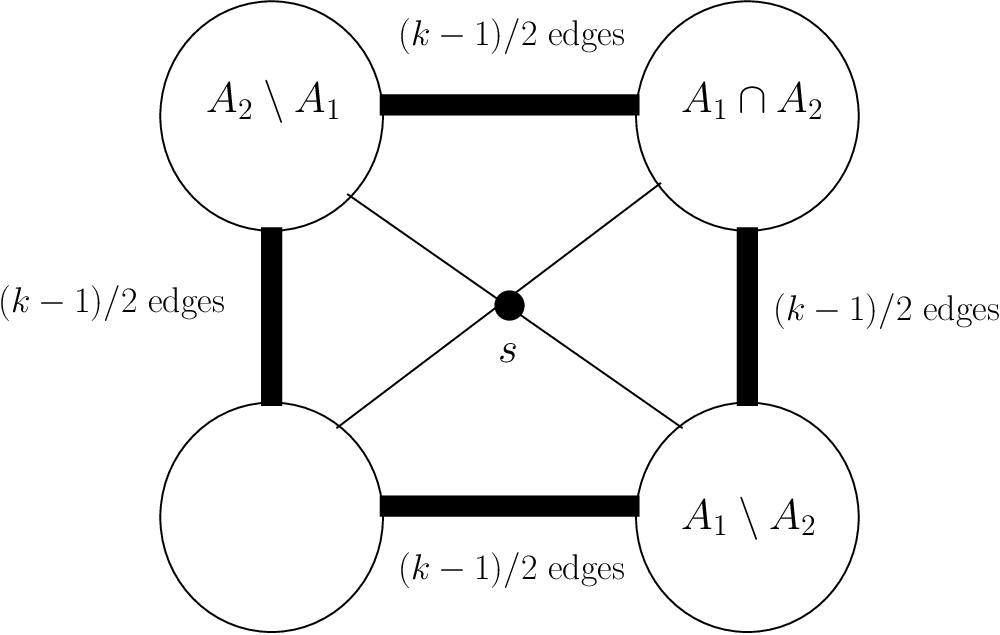}
     \caption{The graph $G$ has this structure if $deg(s)=4$ and two independent sets of size $2$ in $L(G,s,k)$ have a nonempty intersection.}
     \label{degree_4_cycle}
\end{figure}

\begin{lemma}\label{4 cycle}\cite{ORT2016linkages}
Let $G$ be an $(s,k)$-edge-connected graph. If $\deg(s)=4$, then $L(G,s,k)$ is one of: a perfect matching; $C_4$; and $K_4$. If $k$ is even, then $L(G,s,k)$ is not a perfect matching. \newline
Moreover, if $L(G,s,k)$ is a perfect matching, then $G$ has an $(s,(k-1)/2)$-cycle structure whose blobs are the intersections of the dangerous sets corresponding to the maximal independent sets of $L(G,s,k)$.
\end{lemma}

\begin{par}
Note that in case $L(G,s,k)$ is a perfect matching, both the complement of $L(G,s,k)$ and $I(L(G,s,k))$ are a $4$-cycle, and the maximal independent sets of $L(G,s,k)$ are each of size $2$ and they, as well, form a $4$-cycle via intersections. 
\end{par}

\begin{par}
This cyclic structure, for $\deg(s)=4$, is the basis for the results we prove in the coming sections about the structure of $G$ when the complement of its lifting graph is connected. We will generalize this cyclic structure to arbitrary $\deg(s)$ and show that it is one of two possible structures (the other is a path structure) that happen when the maximal independent sets of the lifting graph form a connected entity, i.e. when $I(L(G,s,k))$ is connected. 
\end{par}

\section{Maximal independent sets of the lifting graph}

\begin{par}
We are now ready to study the arrangement of the maximal independent sets of $L(G,s,k)$ and their corresponding dangerous sets in $G$. This is the way we will learn the relation between the structure of a graph and the structure of its lifting graph. We show that there are only two possibilities for $I(L(G, s, k))$ to be connected: it is either a path or a cycle. The latter generalizes the situation for $\deg(s) = 4$. See Figure \ref{degree_4_cycle}. To show this we need first to prove that no three maximal independent sets of $L(G,s,k)$ have a nonempty common intersection. We will first prove the following lemma about three intersecting dangerous sets.
\end{par}

\begin{lemma}\label{three dangerous}
Let $G$ be an $(s,k)$-edge-connected graph such that $k\geq 3$ is odd, and let $A_1,A_2,A_3$ be three distinct dangerous sets such that, for any distinct $i,j\in \{1,2,3\}$:
\begin{enumerate}
\item[(a)] $A_i\cap A_j$ contains at most one neighbour of $s$;
\item[(b)] $|\delta(A_i\cap A_j:A_j\setminus A_i)|=(k-1)/2$;
\item[(c)] $|\delta(A_i\setminus A_j:A_j\setminus A_i)|=0$; and 
\item[(d)] $|\delta_{G-s}(A_i\cap A_j:\overline{A_i\cup A_j})|=0$.
\end{enumerate}
Then $A_1\cap A_2\cap A_3=\emptyset$.
\end{lemma}

\begin{proof}
By way of contradiction, assume that $A_1\cap A_2 \cap A_3\neq \emptyset$.
\begin{claim} \label{at least one is empty}
At least one of the sets $(A_1\cap A_2)\setminus A_3$, $(A_1\cap A_3)\setminus A_2$, and $(A_2\cap A_3)\setminus A_1$ is empty.
\end{claim}
\begin{proof}
\begin{par}
Assume that all of these sets are non-empty. Let (see Figure \ref{intersection_of_three_dangerous_sets})
\begin{itemize}
\item $a= \delta((A_2\cap A_3) \setminus A_1 : A_2\setminus (A_1\cup A_3))$, 
\item $b = \delta((A_1\cap A_2) \setminus A_3 : A_2\setminus (A_1\cup A_3))$, 
\item$c = \delta((A_1\cap A_2) \setminus A_3 : A_1\setminus (A_2\cup A_3))$, 
\item $d = \delta((A_1\cap A_3) \setminus A_2 : A_1\setminus (A_2\cup A_3))$, 
\item $e = \delta((A_1\cap A_3) \setminus A_2 : A_3\setminus (A_1\cup A_2))$, 
\item $f = \delta((A_2\cap A_3) \setminus A_1 : A_3\setminus (A_1\cup A_2))$, 
\item $g = \delta(A_1\cap A_2 \cap A_3 : (A_2\cap A_3)\setminus A_1)$, 
\item $h = \delta(A_1\cap A_2 \cap A_3 : (A_1\cap A_2)\setminus A_3)$, and 
\item $i = \delta(A_1\cap A_2 \cap A_3 : (A_1\cap A_3)\setminus A_2)$. 
\end{itemize}

\begin{figure}[!h]
   \centering
   \includegraphics[scale=0.8]{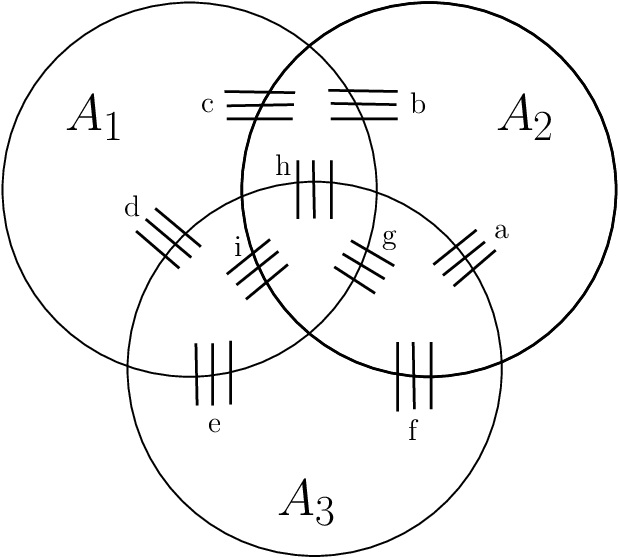} 
   \caption{Intersection of three dangerous sets.}
   \label{intersection_of_three_dangerous_sets}
\end{figure}

These are all the different edge sets between the different parts of the intersections of $A_1$, $A_2$, and $A_3$ since we are given $|\delta(A_i\setminus A_j: A_j\setminus A_i)|= 0$, and $|\delta(A_i\cap A_j: \overline{A_i\cup A_j})|=0$.
\end{par}
\begin{par}
We will show that each one of $g$, $h$, and $i$ is at least $(k-1)/3$ and then show that we have a contradiction, and as a result $(A_1\cap A_2)\setminus A_3$, $(A_1\cap A_3)\setminus A_2$, and $(A_3\cap A_2)\setminus A_1$ cannot be all nonempty.
\end{par}

\begin{par}
If one of $g$, $h$, or $i$ is less than $(k-1)/3$, then by $k$-edge-connectivity, then one of them has to be greater than $(k-1)/3$ because $\delta(A_1\cap A_2\cap A_3)$ contains at most one edge incident with $s$ (hypothesis $(a)$) in addition to the three sets of edges whose sizes are $g$, $h$, and $i$. Suppose without loss of generality that $g$ is less than $(k-1)/3$, and $h$ is greater than $(k-1)/3$. By $(b)$, $g+b=(k-1)/2$ and $h+a=(k-1)/2$, therefore $b$ is greater than $(k-1)/6$ and $a$ is less than $(k-1)/6$.
\end{par}

\begin{par}
Since $(A_2\cap A_3)\setminus A_1$ is nonempty by assumption, $|\delta((A_2\cap A_3)\setminus A_1)|$ has to be at least $k$. 
We must have $f> (k-1)/2$ because $g$ is less than $(k-1)/3$, $a$ is less than $(k-1)/6$, and $\delta((A_2\cap A_3)\setminus A_1)$ contains at most one edge incident with $s$. This is a contradiction to the fact that $f+i=(k-1)/2$ (hypothesis $(b)$).
\end{par}

\begin{par}
Thus each one of $g$, $h$, and $i$ is at least $(k-1)/3$. Then, by $(b)$, $a,b,c,d,e,f$ are each at most $\frac{(k-1)}{6}$. Again, since $(A_2\cap A_3)\setminus A_1$ is nonempty by assumption and contains at most one neighbour of $s$, now $g$ has to be at least $k-(1+\frac{2(k-1)}{6})=\frac{2(k-1)}{3}>\frac{(k-1)}{2}$, a contradiction to $(b)$.
\end{par}

\begin{par}
Thus our assumption that $A_1\cap A_2\cap A_3$, $(A_1\cap A_2)\setminus A_3$, $(A_1\cap A_3)\setminus A_2$, and $(A_3\cap A_2)\setminus A_1$ are all nonempty is false.
\end{par}
\end{proof}

\begin{figure}[!h]
\centering
   \includegraphics[scale=0.8]{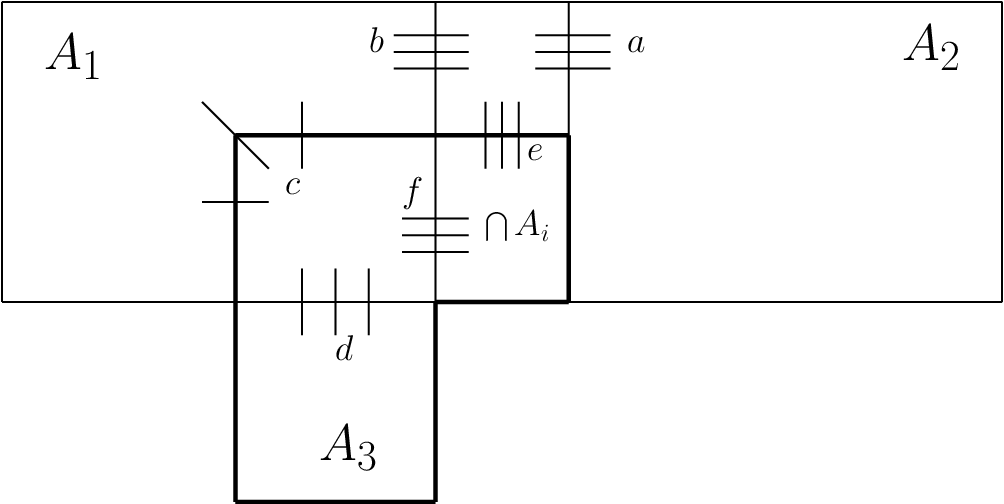} 
   \caption{$(A_3\cap A_2)\setminus A_1$ is empty.}
   \label{A_3_intersect_A_2_delete_A_1_is_empty}
\end{figure}

\begin{claim}
At least two of the sets $(A_1\cap A_2)\setminus A_3$, $(A_1\cap A_3)\setminus A_2$, and $(A_2\cap A_3)\setminus A_1$ are empty.
\end{claim}
\begin{proof}
\begin{par}
 By Claim \ref{at least one is empty}, we may assume $(A_2\cap A_3)\setminus A_1$ is empty and by way of contradiction we assume the other two are not empty. See Figure \ref{A_3_intersect_A_2_delete_A_1_is_empty}.
\end{par}

\begin{par}
Note that there are no edges between $A_1\cap A_2\cap A_3$ and $A_2\setminus A_1$ because $A_2\setminus A_1$ is a subset of $\overline{A_1\cup A_3}$ (as $(A_2\cap A_3)\setminus A_1$ is empty). 
\end{par}
\begin{par}
As illustrated in Figure \ref{A_3_intersect_A_2_delete_A_1_is_empty}, let
\begin{itemize}
\item $a=|\delta((A_1\cap A_2)\setminus A_3: A_2\setminus A_1)|$, \item $b=|\delta((A_1\cap A_2)\setminus A_3: A_1\setminus (A_2\cup A_3))|$, 
\item $c=|\delta((A_1\cap A_3)\setminus A_2:A_1\setminus (A_2\cup A_3))|$, 
\item $d=|\delta((A_1\cap A_3)\setminus A_2:A_3\setminus (A_1\cup A_2))|$, 
\item $e=|\delta(A_1\cap A_2\cap A_3:(A_1\cap A_2)\setminus A_3)|$, and 
\item $f=|\delta(A_1\cap A_2\cap A_3:(A_3\cap A_1)\setminus A_2)|$.
\end{itemize}
\end{par}
\begin{par}
  Note that $\delta(A_1\cap A_2 \cap A_3: A_2\setminus A_3)=\delta(A_1\cap A_2 \cap A_3: (A_1\cap A_2)\setminus A_3)$ ($e$ in Figure \ref{A_3_intersect_A_2_delete_A_1_is_empty}) because $A_1\cap A_2\cap A_3 \subseteq A_1\cap A_3$ and $\delta_{G-s}(A_1\cap A_3 : \overline{A_1 \cup A_3})=\emptyset$.
\end{par}
\begin{par}
If $e<(k-1)/2$, then $f>(k-1)/2$ for $|\delta(A_1\cap A_2\cap A_3)|$ to be at least $k$ since $(A_1\cap A_2\cap A_3)$ contains at most one neighbour of $s$. However, this is impossible because $f+b=(k-1)/2$ (hypothesis $(b)$). Thus, by symmetry, each one of $e$ and $f$ is at least $(k-1)/2$, and so is in fact equal to $(k-1)/2$ and each of $b$ and $c$ is equal to $0$.
\end{par}
\begin{par}
  Since $(A_3\cap A_2)\setminus A_1=\emptyset$, and by hypothesis $(d)$ once applied with $\{i,j\}=\{1,2\}$ and once with $\{i,j\}=\{1,3\}$, we have $\delta(A_1\cap A_3: A_3\setminus A_1)= \delta((A_1\cap A_3)\setminus A_2: A_3\setminus(A_1\cup A_2))$ ($d$ in Figure \ref{A_3_intersect_A_2_delete_A_1_is_empty}) and $\delta(A_1\cap A_2: A_2\setminus A_1)= \delta((A_1\cap A_2)\setminus A_3: A_2\setminus (A_1\cup A_3))$ ($a$ in Figure \ref{A_3_intersect_A_2_delete_A_1_is_empty}). Thus, by hypothesis $(b)$, $a=d=(k-1)/2$.
\end{par}

   \begin{par}
  Since $A_1\cap A_2$ contains at most one neighbour of $s$, either $(A_1\cap A_2)\setminus A_3$ or $A_1\cap A_2 \cap A_3$ does not contain a neighbour of $s$.
  Therefore, either $|\delta((A_1\cap A_2)\setminus A_3)|\leq a+b+e=\frac{k-1}{2}+0+\frac{k-1}{2}=k-1$ or $|\delta(A_1\cap A_2 \cap A_3)|\leq e+f=(k-1)$, a contradiction.
\end{par}
\end{proof}

\begin{figure}[!h]
\centering
   \includegraphics[scale=0.8]{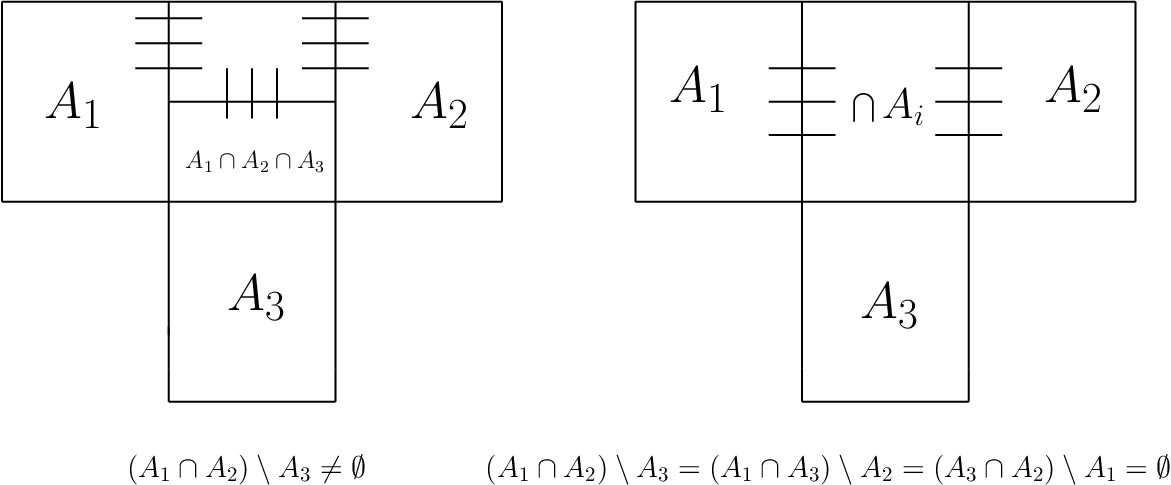} 
   \caption{Only $((A_1\cap A_2)\setminus A_3)$ is possibly nonempty}
   \label{one_part_of_intersection_empty}
\end{figure}

\begin{claim}
All of $(A_1\cap A_2)\setminus A_3$, $(A_1\cap A_3)\setminus A_2$, and $(A_2\cap A_3)\setminus A_1$ are empty.
\end{claim}
\begin{proof}
\begin{par}
  We may assume without loss of generality that only $((A_1\cap A_2)\setminus A_3)$ is nonempty. This is illustrated in the left drawing of Figure \ref{one_part_of_intersection_empty}. There are no edges between $A_1\cap A_2 \cap A_3$ and any set of the form $A_i\setminus (A_j\cup A_l)$, because those will be edges of the sort $\delta_{G-s}(A_j\cap A_l: \overline{A_j\cup A_l})$, and by 
$(d)$ this latter set is empty. Thus the only edges we have coming out of $A_1\cap A_2 \cap A_3$ are to $(A_1\cap A_2)\setminus A_3$. However, since $A_1\cap A_2 \cap A_3= A_2\cap A_3$ in the case we are discussing now, there should be $(k-1)/2$ edges from $A_1\cap A_2 \cap A_3$ to $A_3\setminus A_2$, but $A_3\setminus A_2= A_3\setminus (A_1\cup A_2)$, a contradiction.
\end{par}
\end{proof}

\begin{par}
Now, $(A_1\cap A_2)\setminus A_3$, $(A_1\cap A_3)\setminus A_2$, and $(A_3\cap A_2)\setminus A_1$ are all empty (right drawing in Figure \ref{one_part_of_intersection_empty}). By $(b)$ there are $(k-1)/2$ edges from $A_1\cap A_2=A_1\cap A_2\cap A_3$ to $A_2\setminus A_1$. Those are also edges from $A_1\cap A_3$ to $\overline{A_1\cup A_3}$, and so we have a contradiction because $\delta_{G-s}(A_1\cap A_3:\overline{A_1\cup A_3})=\emptyset$ and $(k-1)/2\geq 1$ as $k\geq 3$.
\end{par} \end{proof}

Now we project the information we found on dangerous sets to independent sets and prove that no three distinct maximal independent sets have a nonempty common intersection. Note that in \cite{Bang-Jensen2000splitting-specified-subset} and \cite{Bang-Jensen2004splitting-between-two-subset} Jord\'an and Beng-Jensen showed that $s$ cannot have a neighbour in the intersection of three maximal dangerous sets, however, they did not show that the intersection itself is empty, but we will show this for dangerous sets too in Lemma \ref{cycle of three dangerous sets} below.

\begin{lemma}\label{common intersection}
Let $G$ be an $(s,k)$-edge-connected graph such that $k\geq 2$ and $\deg(s)\geq 4$. Then no three distinct maximal independent sets of $L(G,s,k)$ have a nonempty common intersection.
\end{lemma}
\begin{proof}

 \begin{par}
  By way of contradiction, assume that there are three maximal independent sets $I_1$, $I_2$, and $I_3$ in $L(G,s,k)$ such that $I_1\cap I_2 \cap I_3\neq \emptyset$. Then each of them is of size at least $2$. Let $A_1$, $A_2$, and $A_3$ respectively be corresponding dangerous sets in $G$. By Lemma \ref{two dangerous}, $|I_1\cap I_2 \cap I_3|=1$, so $s$ has exactly one neighbour in $A_1\cap A_2 \cap A_3$. Also, by the previous results, we may assume that $k$ is odd. 
\end{par}

\begin{par}
  We may assume that at most one of $I_1$, $I_2$, and $I_3$ has size $\lceil deg(s)/2 \rceil$: If two intersecting independent sets have this size, then by $(3)$ in Theorem \ref{summary of lifting graph facts}, in case $\deg(s)>4$, $L(G,s,k)$ is an isolated vertex plus a balanced complete bipartite graph, and there does not exist three distinct maximal independent sets, only two. If $\deg(s)=4$, then $L(G,s,k)$ is a perfect matching and its complement is a $4$-cycle and it has four maximal independent sets of size $2$ each, forming a cycle, Lemma \ref{4 cycle}. In this case no three distinct maximal independent sets have a nonempty intersection
  \end{par}
  \begin{par}
  By Lemma \ref{two dangerous}, $k$ is odd since two independent sets not both of size $\lceil \deg(s)/2 \rceil$ have a nonempty intersection, and we have $|\delta(A_i\cap A_j:A_i\setminus A_j)| = |\delta(A_i\cap A_j:A_j\setminus A_i)| = (k-1)/2$ for every distinct $i$ and $j$ in $\{1,2,3\}$.
 \end{par}

\begin{par}
Also from Lemma \ref{two dangerous} we know that for every distinct $i$ and $j$ in $\{1,2,3\}$ we have, $|I_i\cap I_j|= 1$, $|\delta_{G-s}(A_i\cap A_j: \overline{A_i\cup A_j})|=0$, and $|\delta(A_i\setminus A_j: A_j\setminus A_i)|=0$. In particular, the only sets of edges between the different parts of the intersections of $A_1$, $A_2$, and $A_3$ are the ones illustrated in Figure \ref{intersection_of_three_dangerous_sets}. All the conditions of Lemma \ref{three dangerous} are satisfied, therefore, $\bigcap\limits^3_{i=1} A_i =\emptyset$. This is a contradiction since $\bigcap\limits^3_{i=1} A_i$ contains the non-$s$ end-vertex of the edge common between the three maximal independent sets.
\end{par} \end{proof}

The following lemma tells us that for each pair of disjoint maximal independent sets in $L(G,s,k)$ there exist a pair of disjoint dangerous sets corresponding to them in $G$.

\begin{lemma}\label{disjoint dangerous sets}
Let $G$ be an $(s,k)$-edge-connected graph such that $k\geq 2$ and $\deg(s)\geq 4$. If $I_1$ and $I_2$ are disjoint maximal independent sets in $L(G,s,k)$ of size at least $2$ each, and $A_1$ and $A_2$ are corresponding dangerous sets in $G$, then either $A_1\setminus A_2$ or $A_2\setminus A_1$ is dangerous. 
\end{lemma}
\begin{proof}
\begin{par}
Consider two disjoint maximal independent sets $I_1$ and $I_2$ of $L(G,s,k)$ and let $A_1$ and $A_2$ be corresponding dangerous sets in $G$ such that $A_1\cap A_2 \neq \emptyset$. Define $k_1$ to be $|\delta(\{s\}:A_1\setminus A_2)|$, $k_2=|\delta(\{s\}:A_2\setminus A_1)|$, and $k_3=|\delta(\{s\}:A_1\cap A_2)|$. By the assumption of the independent sets being disjoint, $k_3=0$, so $k_1$ and $k_2$ are the sizes of the two independent sets.
\end{par}
\begin{par}
Now, in $G-s$,
\begin{enumerate}
  \item[(1)]$|\delta_{G-s}(A_1)|\leq (k+1)-k_1$,
  \item[(2)]$|\delta_{G-s}(A_2)|\leq (k+1)-k_2$, 
  \item[(3)]$|\delta_{G-s}(A_1\cap A_2)|\geq k$, \item[(4)]$|\delta_{G-s}(A_2\setminus A_1)|\geq k-k_2$, \item[(5)]$|\delta_{G-s}(A_1\setminus A_2)|\geq k-k_1$, and 
  \item[(6)]if $\overline{A_1\cup A_2 \cup \{s\}}\neq \emptyset$, then $|\delta_{G-s}(\overline{A_1\cup A_2})|\geq (k+2)-(k_1+k_2)$ because $A_1\cup A_2$ is not dangerous.
\end{enumerate}
\end{par}

\begin{par}
  In the first case, $\overline{A_1\cup A_2 \cup \{s\}}= \emptyset$. Then either $k_1=k_2=\deg(s)/2$, in case $\deg(s)$ is even, or one of $k_1$ and $k_2$, say $k_1$, is $\frac{(\deg(s)+1)}{2}$ and the other is $\frac{(\deg(s)-1)}{2}$ in case $\deg(s)$ is odd (Frank's theorem \ref{Frank}). Since $\overline{A_1\cup A_2 \cup \{s\}}= \emptyset$, we have,
  
  $$|\delta(A_1)|=k_1+|\delta(A_1\cap A_2: A_2\setminus A_1)|+|\delta(A_1\setminus A_2:A_2\setminus A_1)|$$ 
  and 
  $$|\delta(A_2\setminus A_1)|=k_2+|\delta(A_1\cap A_2: A_2\setminus A_1)|+|\delta(A_1\setminus A_2:A_2\setminus A_1)|.$$
  
  Thus, $0\leq |\delta(A_1)|-|\delta(A_2\setminus A_1)|\leq 1$, so $A_2\setminus A_1$ is dangerous, and we are done.
\end{par}

\begin{par}
  We will show that, in case $\overline{A_1\cup A_2 \cup \{s\}}\neq \emptyset$, either $\delta_{G-s}(A_1\setminus A_2)\leq (k+1)-k_1$ or $\delta_{G-s}(A_2\setminus A_1)\leq (k+1)-k_2$, i.e. $\delta_{G}(A_1\setminus A_2)\leq (k+1)$ or $\delta_{G}(A_2\setminus A_1)\leq (k+1)$. That is either $(A_1\setminus A_2)$ or $(A_2\setminus A_1)$ is dangerous in $G$. Then we can replace $A_1$ by $A_1\setminus A_2$ or replace $A_2$ by $A_2\setminus A_1$ and consequently have disjoint dangerous sets.
\end{par}
\begin{par}
 In the remaining case, $\overline{A_1\cup A_2 \cup \{s\}}\neq \emptyset$.
 Let $\varepsilon_i=|\delta_{G-s}(A_i\setminus A_{3-i})|-(k-k_i)$, for $i=1,2$. Then $(1)-(5)$ applied to the standard equation \ref{intersection of two cuts} in $G-s$ gives,
  
    $$4k-2(k_1+k_2)+2+\varepsilon_1+\varepsilon_2\leq RHS=LHS\leq 4k-2(k_1+k_2)+4$$
  
  Therefore $\varepsilon_1+\varepsilon_2\le 2$, so, for some $i\in\{1,2\}$, $\varepsilon_i\le 1$.  For such an $i$, $A_i\setminus A_{3-i}$ is dangerous, as required.
  \end{par}\end{proof}

\begin{lemma} \label{minimal dangerous}
Let $G$ be an $(s,k)$-edge-connected graph such that $k\geq 2$ and $\deg(s)\geq 4$. If $A_1$ and $A_2$ are the dangerous sets in $G$ corresponding to two disjoint maximal independent sets $I_1$ and $I_2$ in $L(G,s,k)$, then $A_1$ and $A_2$ are disjoint.
\end{lemma}
\begin{proof}
Recall from Definition \ref{the corresponding dangerous} that $A_1$ and $A_2$ are, in fact, the unique minimal dangerous sets in $G$ corresponding to $I_1$ and $I_2$. By Lemma \ref{disjoint dangerous sets}, since $I_1$ and $I_2$ are disjoint, either $A_1\setminus A_2$ or $A_2\setminus A_1$ is dangerous. If $A_1\cap A_2\neq \emptyset$, then this contradicts the minimality of either $A_1$ or $A_2$.
\end{proof}

\begin{remark}
By taking the unique minimal dangerous sets in $G$ corresponding to the maximal independent sets of $L(G,s,k)$, we have the path and cycle structures of the components of $I(L(G,s,k))$ reflected in the arrangement of these dangerous sets as will be seen in Lemmas \ref{cycle of three dangerous sets} and \ref{component of max indep sets}.
\end{remark}

\begin{lemma}\label{cycle of three dangerous sets}
 Let $G$ be an $(s,k)$-edge-connected graph such that $k\geq 2$ and $\deg(s)\geq 4$. Then for any three distinct maximal independent sets $I_1,I_2,I_3$ in $L(G,s,k)$, of size at least $2$ each, the corresponding dangerous sets $A_1,A_2,A_3$ in $G$ satisfy $A_1\cap A_2\cap A_3=\emptyset$. In particular, if $I_1,I_2,I_3$ form a cycle in $I(L(G,s,k))$, then for $\{i,j,l\}=\{1,2,3\}$, $A_i$ intersects $A_j\cup A_l$ exactly in $A_j\setminus A_l$ and $A_l\setminus A_j$, see Figure \ref{cycle_of_three_dangerous_sets}.
\end{lemma}

\begin{figure}[!h]
\centering
   \includegraphics[scale=0.45]{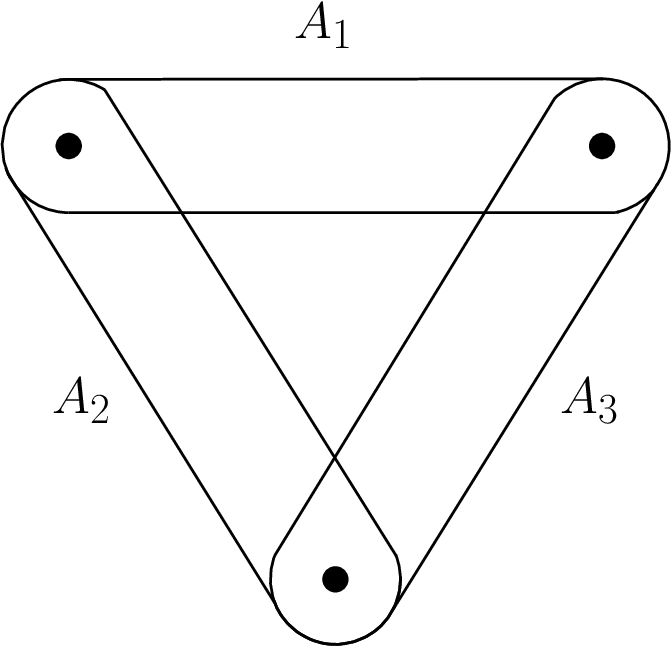} 
   \caption{Cycle of three dangerous sets.}
   \label{cycle_of_three_dangerous_sets}
\end{figure}

\begin{proof}
\begin{par}
By Lemma \ref{common intersection} $I_1\cap I_2\cap I_3 = \emptyset$.
Up to symmetry there are four cases to consider:
\begin{itemize}
\item [(1)] $I_1\cap I_2 \neq \emptyset$, $I_2\cap I_3 \neq \emptyset$, but $I_1\cap I_3= \emptyset$, that is $I_1I_2I_3$ is an induced path in $I(L(G,s,k))$,
\item [(2)] $I_1\cap I_2\neq \emptyset$ but $I_3$ is disjoint from $I_1\cup I_2$,
\item [(3)] $I_1$, $I_2$, and $I_3$ are pairwise disjoint,
\item [(4)] $I_1$, $I_2$, and $I_3$ make a cycle in $I(L(G,s,k))$.
\end{itemize}
\end{par}
\begin{par}
By Lemma \ref{minimal dangerous}, in cases $(1)$ - $(3)$, at least two of $A_1$, $A_2$, and $A_3$ are disjoint. Then $A_1\cap A_2\cap A_3=\emptyset$. Now suppose that $I_1,I_2,I_3$ form a cycle in $I(L(G,s,k))$. By Lemma \ref{two dangerous}, each $I_i$ and $I_j$ intersect in exactly one vertex, so $s$ has exactly one neighbour in $A_i\cap A_j$.
\end{par}
 \begin{par}
   For distinct $i,j \in \{1,2,3\}$, $I_i\cup I_j\neq V(L(G,s,k))$ because, by $(2)$ and $(3)$ in Theorem \ref{summary of lifting graph facts}, the only case in which the union of two intersecting maximal independent sets of $L(G,s,k)$ equals $V(L(G,s,k))$ is the case when $L(G,s,k)$ is an isolated vertex plus a complete balanced bipartite graph, and there are only two maximal independent sets in $L(G,s,k)$ in that case, whereas here we have at least three maximal independent sets. Thus, by Lemma \ref{two dangerous}, hypotheses $(a)-(d)$ of Lemma \ref{three dangerous} hold for any distinct $i,j \in \{1,2,3\}$. Therefore, $A_1\cap A_2\cap A_3=\emptyset$.\end{par}\end{proof}

The following proposition describes the situation when a dangerous set is intersected by two other dangerous sets (see Figure \ref{dangerous_set_between_two_dangerous_sets}). Compare this to \cite[Lemma~3.3]{Jordan1999ConstrainedSplitting}, which is proved under the assumption that $\deg(s)$ is even.

\begin{lemma}\label{dangerous set between two other dangerous sets}
Let $G$ be an $(s,k)$-edge-connected graph such that $k\geq 2$ and $\deg(s)\geq 4$. Let $A_1,A_2,A_3$ be, in order, dangerous sets corresponding to consecutive vertices $I_1,I_2,I_3$ in a path or $3$-cycle in $I(L(G,s,k))$. Then (see Figure \ref{dangerous_set_between_two_dangerous_sets}),
\begin{enumerate}

\item[(1)] for $i= 1,3$, $|\delta(\{s\}: A_i\cap A_2)|=1$;
\item[(2)] $\delta(A_2)= \delta(A_2\cap A_1: A_1\setminus A_2) \cup \delta(\{s\}: A_2) \cup \delta(A_2\cap A_3: A_3\setminus A_2)$;
\item[(3)] $|I_2|=2$; 
\item[(4)] $A_2\setminus (A_1\cup A_3)=\emptyset$; and
\item[(5)] $|\delta(A_1\cap A_2 : A_2\cap A_3)|= (k-1)/2$.

\end{enumerate}

\end{lemma}

\begin{figure}[!h]
\centering
   \includegraphics[scale=0.6]{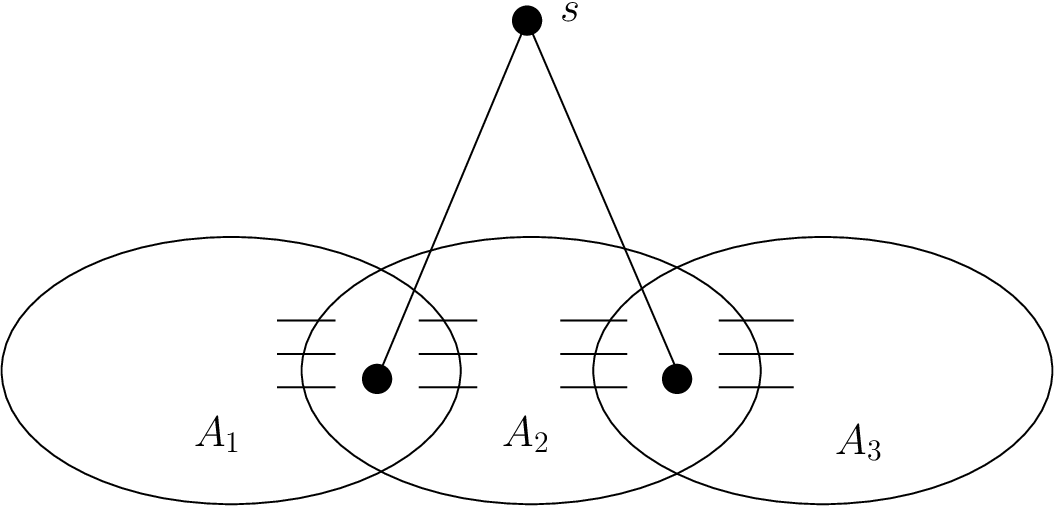} 
   \caption{Dangerous set between two dangerous sets}
   \label{dangerous_set_between_two_dangerous_sets}
\end{figure}

\begin{proof}
\begin{par}
  Let $A_1$, $A_2$, and $A_3$ be three dangerous sets as described above. See Figure \ref{dangerous_set_between_two_dangerous_sets}. By $(2)$ and $(3)$ in Theorem \ref{summary of lifting graph facts}, the only case in which the union of two intersecting maximal independent sets is equal to $V(L(G,s,k))$ is when $L(G,s,k)$ is an isolated vertex plus a balanced complete bipartite graph, and there are exactly two maximal independent sets in that case, not three. Thus, regarded as maximal independent sets, $I_2\cup I_3 \neq V(L(G,s,k))$ and $I_2\cup I_1 \neq V(L(G,s,k))$. Therefore Lemma \ref{two dangerous} applies to $A_2$ with $A_1$, and with $A_3$.  
  \end{par}
  \begin{par}
  In particular we have, $|\delta(A_1\cap A_2:A_1\setminus A_2)|=|\delta(A_3\cap A_2:A_3\setminus A_2)|=(k-1)/2$, and $|\delta(\{s\}:A_1\cap A_2)|=|\delta(\{s\}:A_3\cap A_2)|=1$. This is a total of $1+1+\frac{(k-1)}{2}+\frac{(k-1)}{2}=k+1$ edges in $\delta(A_2)$. There cannot be more since $A_2$ is dangerous.
  \end{par}
  \begin{par}
  Thus, $\delta(A_2)= \delta(A_2\cap A_1: A_1\setminus A_2) \cup \delta(\{s\}: A_2) \cup \delta(A_2\cap A_3: A_3\setminus A_2)$, and the two edges from $s$ to $A_1\cap A_2$ and $A_2\cap A_3$ are the only edges between $s$ and $A_2$, i.e. $|I_2|=2$.
  Moreover, $|\delta_{G-s}(A_2:\overline{A_1\cup A_2 \cup A_3})|=0$ and $|\delta(\{s\}:A_2\setminus (A_1\cup A_3))|=0$.
\end{par}

\begin{par}
  Now we show that $A_2\setminus (A_1\cup A_3)=\emptyset$. By way of contradiction suppose not. By Lemma \ref{two dangerous}, $|\delta(A_2\setminus (A_1\cup A_3): (A_1\setminus A_2)\cup (A_3\setminus A_2))|=0$. Since we also have $|\delta_{G-s}(A_2:\overline{A_1\cup A_2 \cup A_3})|=0$ and $|\delta(\{s\}:A_2\setminus (A_1\cup A_3))|=0$, it follows that 
  $$|\delta(A_2\setminus (A_1\cup A_3))|=|\delta(A_2\setminus (A_1\cup A_3):A_1\cap A_2)|+|\delta(A_2\setminus (A_1\cup A_3):A_3\cap A_2)|.$$ 
  This is less than or equal to $|\delta(A_2\setminus A_1:A_1\cap A_2)|+|\delta(A_2\setminus A_3:A_3\cap A_2)|=\frac{(k-1)}{2}+\frac{(k-1)}{2}=k-1$, a contradiction since $G$ is $(s,k)$-edge-connected. Thus, $A_2\setminus (A_1\cup A_3)=\emptyset$.
  \end{par}
  \begin{par}
  Now, since $A_2\setminus (A_1\cup A_3)=\emptyset$, $A_2\setminus A_1 = A_2\cap A_3$. Thus from Lemma \ref{two dangerous}, part (2)-(d), we have  $|\delta(A_1\cap A_2 : A_2\cap A_3)|= (k-1)/2$.
\end{par} \end{proof}

Using the previous lemmas we can now prove that the maximal independent sets of $L(G,s,k)$ are arranged in paths and cycles, and the same arrangements occurs for the corresponding dangerous sets.

\begin{lemma}\label{component of max indep sets}
Let $G$ be an $(s,k)$-edge-connected graph such that $k\geq 2$ and $\deg(s)\geq 4$. Then every component of $I(L(G,s,k))$ is either a path (possibly a singleton) or a cycle. Moreover, the corresponding dangerous sets in $G$ respectively make a path or a cycle via intersections.
\end{lemma}
\begin{proof}
\begin{par}
We will show that $I(L(G,s,k))$ has no vertex of degree at least 3. Therefore, each component of $I(L(G,s,k))$ is either a cycle or a path. Suppose for a contradiction that $I(L(G,s,k))$ contains a vertex of degree $3$. Then there are maximal independent sets $I_1, I_2, I_3, I_4$ in $L(G,s,k)$ such that $I_1$ intersects each one of $I_2$, $I_3$, and $I_4$. By Lemma \ref{common intersection}, no three of $I_2$, $I_3$, and $I_4$ have a nonempty common intersection, therefore, $I_4$ intersects $I_1$ in $I_1\setminus (I_2\cup I_3)$. Let $A_1$, $A_2$, $A_3$, and $A_4$ be the minimal corresponding dangerous sets. Then the non-$s$ end of the unique edge in $I_4\cap (I_1\setminus (I_2\cup I_3))$ belongs to $(A_4\cap A_1)\setminus (A_2\cup A_3)$, contradicting Lemma \ref{dangerous set between two other dangerous sets} as $A_1\setminus (A_2\cup A_3)$ should be empty.
\end{par}
\begin{par}
   The dangerous sets corresponding to non-adjacent vertices of $I(L(G,s,k))$ are disjoint by Lemma \ref{minimal dangerous}. Therefore, this collection of dangerous sets corresponding to a component of $I(L(G,s,k))$ will form either a path or a cycle via intersections. Each one of those intersections contains exactly one neighbour of $s$ (but may contain other vertices).\newline 
   In particular, if a component of $I(L(G,s,k))$ is a $3$-cycle with vertices $I_1$, $I_2$, and $I_3$, then by Lemma \ref{cycle of three dangerous sets} the corresponding dangerous sets $A_1$, $A_2$, and $A_3$ satisfy $A_1\cap A_2 \cap A_3=\emptyset$, so they form a $3$-cycle of dangerous sets.
\end{par} \end{proof}

The following corollary proves the same structures for the components of the complement of $L(G,s,k)$ as in \cite[Theorem~3.4]{Jordan1999ConstrainedSplitting}, however here we generalize this to arbitrary parity of $\deg(s)$.

\begin{corollary} \label{component of complement}
Let $G$ be an $(s,k)$-edge-connected graph such that $k\geq 2$ and $\deg(s)\geq 4$. Then every component of the complement of $L(G,s,k)$ is one of, a cycle, a clique, or two cliques with a path between them.
\end{corollary}
\begin{proof}
By Lemma \ref{component of max indep sets} every component of $I(L(G,s,k))$ is either a path (possibly a singleton) or a cycle. The maximal independent sets of $L(G,s,k)$ corresponding to the vertices of the cycle and the interior vertices of the path are of size $2$ each by Lemma \ref{dangerous set between two other dangerous sets}. Thus a cycle component of $I(L(G,s,k))$ gives a cycle component in the complement of $L(G,s,k)$ (because the complement of a maximal independent set of size $2$ consists of one edge). Now consider a path component of $I(L(G,s,k))$. If it is a singleton, then in the complement of $L(G,s,k)$ it gives a clique (the complement of one maximal independent set). If the path component of $I(L(G,s,k))$ contains at least two vertices, then in the complement of $L(G,s,k)$, we have two cliques corresponding to the first an last maximal independent sets (possibly of different sizes) in the path. These two cliques either meet at a vertex (if the path in $I(L(G,s,k))$ consists of one edge), or there is a nontrivial path between the two cliques corresponding to the path of maximal independent sets of size $2$ each. 
\end{proof}

From Lemma \ref{component of max indep sets} we know that each component of $I(L(G,s,k))$ is either a path or a cycle, thus any spanning tree of any component is a path. In the following lemma we show that, other than the edges incident with $s$, there are no edges coming out from the cycle of dangerous sets in $G$ corresponding to a cycle component of $I(L(G,s,k))$, and the edges coming out from a path of dangerous sets only come out from the first and last dangerous sets in the path.

\begin{lemma}\label{edges out from union of dangerous sets}
 Let $G$ be an $(s,k)$-edge-connected graph such that $k\geq 2$ and $\deg(s)\geq 4$. If $A_1, \cdots, A_m$ are, in order, the dangerous sets corresponding to a spanning path of a non-singleton component $C$ of $I(L(G,s,k))$, then 
\begin{enumerate}
\item[(1)] $\delta_{G-s}(A_1 \cup \cdots \cup A_m) = \emptyset$ if $C$ is a cycle, and 

\item[(2)] $\delta_{G-s}(A_1 \cup \cdots \cup A_m) = \delta_{G-s}((A_1\setminus A_2) \cup (A_m\setminus A_{m-1}): \overline{A_1 \cup \cdots \cup A_m})$ if $C$ is a path.
\end{enumerate}
\end{lemma}
\begin{proof}
\begin{par}
  The dangerous sets in a cycle or the interior of a path are saturated. That is we have already identified $k+1$ edges incident with each one of them, as follows. If $A_2$ is between $A_1$ and $A_3$, then there are $(k-1)/2$ edges to $A_1\setminus A_2$, $(k-1)/2$ edges to $A_3\setminus A_2$, one edge from $s$ to $A_2\cap A_1$, and one edge from $s$ to $A_2\cap A_3$. Thus if the component $C$ is a cycle, then the corresponding cycle of dangerous sets in $G$ does not have any edges to vertices outside their union except to $s$. If $C$ is a path, then only the first and last sets $A_1$ and $A_m$ of the corresponding path of dangerous sets in $G$ can have edges to vertices in $\overline{A_1 \cup \cdots \cup A_m}$ different from $s$.
  \end{par} \end{proof}

  The following lemma analyses the situation when all the components of $I(L(G,s,k))$ are cycles except possibly one that is a path. This will be needed in proving the path or cycle structure of $G$ when the complement of $L(G,s,k)$ is connected.

\begin{lemma}\label{vertices outside dangerous sets}
Let $G$ be an $(s,k)$-edge-connected graph such that $k\geq 2$ and $\deg(s)\geq 4$. Let $\mathcal{D}$ be the collection of the unique minimal dangerous sets in $G$ corresponding to the maximal independent sets of $L(G,s,k)$ of size at least $2$, and let $B=V(G)\setminus(\{s\}\cup \bigcup \mathcal{D}))$. We have the following:
\begin{enumerate}
\item[(1)] $B=\emptyset$ in the following cases:
           \begin{itemize}
           \item all the components of $I(L(G,s,k))$ are cycles or,
            \item exactly one component of $I(L(G,s,k))$ is a path such that in case it is a singleton the maximal independent set corresponding to it in $L(G,s,k)$ is not a singleton.
           \end{itemize}
\item[(2)] 
Suppose that exactly one component of $I(L(G,s,k))$ is a path, it is not a singleton, and let $A_1,\cdots,A_m$ be the corresponding path of dangerous sets in $G$. Then,
    \begin{enumerate}
    \item[(a)] $\delta(A_1\cup \cdots \cup A_m) \subseteq \delta(\{s\})$,
    \item[(b)]$\delta(A_1)= \delta(\{s\}: A_1) \cup \delta(A_1\cap A_2 : A_2\setminus A_1) \cup \delta(A_1\setminus A_2: A_m\setminus A_{m-1})$, and
    \item[(c)] $\delta(A_m)= \delta(\{s\}: A_m) \cup \delta(A_m\cap A_{m-1} : A_{m-1}\setminus A_m) \cup \delta(A_m\setminus A_{m-1}: A_1\setminus A_2)$.
    \end{enumerate}
\end{enumerate}
\end{lemma}
\begin{proof}
\begin{par}
  Suppose for a contradiction that $B\neq \emptyset$, while at most one component of $I(L(G,s,k))$ is a path, and in case such a component exists and is a singleton the maximal independent set corresponding to it in $L(G,s,k)$ is not a singleton. Then $|\delta(B)|\geq k$, and none of the maximal independent sets of $L(G,s,k)$ is a singleton. Thus the neighbours of $s$ are all in $\bigcup \mathcal{D}$ as this is the union of the dangerous sets corresponding to the maximal independent sets of size at least $2$. Therefore, $\delta(B:\{s\})=\emptyset$, i.e. $\delta(B)=\delta(B:\bigcup \mathcal{D})$.
  \end{par}
\begin{par}
  If $C$ is a cycle component of $I(L(G,s,k))$ and $A_1,\cdots,A_n$ is the corresponding cycle of dangerous sets (Lemma \ref{component of max indep sets}), then $\delta_{G-s}(A_1\cup \cdots \cup A_n)=\emptyset$ (Lemma \ref{edges out from union of dangerous sets}). Thus for any such component $\delta(B:A_1\cup \cdots \cup A_n)=\emptyset$. Therefore, if all the components of $I(L(G,s,k))$ are cycles, $\delta(B)=\emptyset$ and we have a contradiction. Hence $B$ must be empty in this case. 
\end{par}
 \begin{par}
    Now let us consider the case when exactly one of the components is a path. Let $C$ be such a component and assume first that it is a singleton. By assumption, the maximal independent set $I$ corresponding to $C$ in $L(G,s,k)$ is not a singleton. Therefore, if $A$ is the dangerous set in $G$ corresponding to $I$, then $|\delta(\{s\}:A)|\geq 2$. Since $A$ is dangerous, we therefore have $|\delta(B:A)|\leq k-1$. Note that $\delta(B)= \delta(B:A)$ since $C$ is the only non-cycle component of $I(L(G,s,k))$. Thus we have a contradiction, so $B$ is empty in that case too.
    
 \end{par}   
\begin{par}
    Suppose that $C$ is a path component that is not a singleton, and let $A_1,\cdots , A_m$ be the corresponding path of dangerous sets (Lemma \ref{component of max indep sets}). Then $m\geq 2$ and $\delta(B)$ is equal to $\delta(B:A_1\cup \cdots \cup A_m)$. Again by Lemma \ref{edges out from union of dangerous sets}, if $m\geq 3$, then $\delta(B:A_i)=\emptyset$ for any $1<i<m$. Thus, for $m\geq 2$, $\delta(B)=\delta(B:A_1 \cup A_m)$. 
  \end{par}
  \begin{par}
  If $I(L(G,s,k))$ is a path of length $1$, i.e. $m=2$ and $L(G,s,k)$ consists of two maximal independent sets intersecting in a vertex, then by part $(1)$ in Lemma \ref{large independent}, $\overline{\{s\}\cup A_1 \cup A_2} = \emptyset$, and points $(a)-(c)$ hold in that case. Thus, we may assume that if $I(L(G,s,k))$ is equal to its unique path component, then it is a path of length at least two. 
   \end{par}
  \begin{par}
  Therefore, in any case, whether or not $I(L(G,s,k))$ is equal to its unique path component, we may assume that $L(G,s,k)$ contains at least three maximal independent sets. By $(2)$ and $(3)$ in Theorem \ref{summary of lifting graph facts}, this means that for any two maximal independent sets in $L(G,s,k)$ with a nonempty intersection, their union is not equal to $V(L(G,s,k))$. Hence part $(2)$ of Lemma \ref{two dangerous} applies here and we have $|\delta(A_{m-1}\cap A_m : A_{m-1}\setminus A_{m})|=\frac{(k-1)}{2}$ and $|\delta(A_2\cap A_1 : A_2\setminus A_1)|=\frac{(k-1)}{2}$. There is one edge from $s$ to $A_{m-1}\cap A_m$, and one edge from $s$ to $A_2\cap A_1$, and there is at least one edge from $s$ to $A_1\setminus A_2$ and at least one edge to $A_{m}\setminus A_{m-1}$. Since $A_1$ and $A_m$ are dangerous, both $|\delta(A_1: B)|$ and $|\delta(A_m: B)|$ are at most $(k+1)-(1+1+\frac{(k-1)}{2})=\frac{(k-1)}{2}$. This means that $|\delta(B)|\leq k-1$, a contradiction. Thus, $B=\emptyset$.
  \end{par}
  \begin{par}
     Recall that by Lemma \ref{edges out from union of dangerous sets}, there are no edges coming out of the cycles of dangerous sets corresponding to the cycle components of $I(L(G,s,k))$, so specifically no edges from them to $A_1\cup A_m$. Therefore, if $m=2$, then $(a)-(c)$ hold. If $m\geq 3$, then by $(2)$ in Lemma \ref{dangerous set between two other dangerous sets}, for any $i\in \{2,\cdots, m-1\}$, $\delta(A_i)= \delta(\{s\}: A_i) \cup \delta(A_i\cap A_{i+1}: A_{i+1}\setminus A_i) \cup \delta(A_i\cap A_{i-1}: A_{i-1}\setminus A_i)$, hence, $\delta(A_1:A_i)=\emptyset$ for any $i\neq 2$ in $\{2,\cdots, m-1\}$, if exists, and $\delta(A_m:A_i)=\emptyset$ for any $i\neq m-1$ in $\{2,\cdots, m-1\}$, if exists. Thus, $(a)-(c)$ hold in that case too.
  \end{par}
\end{proof}

\begin{remark}
\begin{par}
Suppose that exactly one component of $I(L(G,s,k))$ is a path, it is a singleton, and the corresponding maximal independent set in $L(G,s,k)$ is a singleton, and let $sv$ be the edge that this maximal independent set consists of. Then $sv$ is liftable with every other edge incident with $s$, including the edges parallel to it, if any. Thus, if $s$ has a neighbour other than $v$, then $sv$ should have multiplicity at least $k+2$, otherwise either $s$ is incident with a cut-edge or lifting two parallel edges incident with $s$ results in a cut smaller than $k$ with vertices different from $s$ on both sides. In case such a component exists, the set $B=V(G)\setminus(\{s\}\cup \bigcup \mathcal{D}))$ does not have to be empty, but should have at least $k$ edges to $v$ (as it cannot have edges to the cycles of dangerous sets corresponding to the other components). Note that none of the dangerous sets in $\mathcal{D}$ contains $v$ since these dangerous sets correspond only to maximal independent sets of size at least $2$. If more than one component of $I(L(G,s,k))$ is a path, then the first and last sets in the path of dangerous sets corresponding to one component can have edges to the first and last sets in the path of dangerous sets corresponding to another component, as well as edges to $B$, which does not have to be empty in that case because the total number of edges coming to it from the different paths of dangerous sets can be bigger than $k-1$.
\end{par}
\end{remark}

\section{Structure of the lifting graph when its complement is connected}

\begin{par}
In this section we investigate the structure of the lifting graph when its complement is connected, equivalently when $I(L(G,s,k))$ is connected, i.e. the maximal independent sets of the lifting graph form a connected entity via intersctions. One possibility is that the complement of the lifting graph is a Hamilton cycle, hence generalizing what Ok, Richter, and Thomassen proved at $\deg(s)=4$ (Lemma \ref{4 cycle}). The other possibility is that the complement is two cliques of the same size with a path between them (possibly of length $0$ as in the case when $L(G,s,k)$ is an isolated vertex plus a balanced complete bipartite graph). This generalizes an example of Ok, Richter, and Thomassen for $\deg(s)=6$, where the lifting graph was $K_{3,3}$ minus an edge. Illustrated in Fig. 4 of \cite{ORT2016linkages} it can be seen that the maximal independent sets of the lifting graph are of sizes $3$, $2$, and $3$ and they form a path via intersections (the set of size $2$ is in the middle and it is the complement of the missing edge). This means the complement is two cliques of size $3$ with a path of length one (the lost edge of $K_{3,3}$) between them.
\end{par}

\begin{theorem} (Connected complement of the lifting graph) \label{connected complement}
Let $G$ be an $(s,k)$-edge-connected graph such that $k\geq 2$ and $\deg(s)\geq 4$. If $I(L(G,s,k))$ is connected, then: 
 
 \begin{enumerate} 
 \item [(a)]  $I(L(G,s,k))$ is either a cycle or a path; 
   \item [(b)] any vertex of $I(L(G,s,k))$ of degree $2$ corresponds to a maximal independent set of size $2$;

 \item [(c)] if $I(L(G,s,k))$ is a cycle, then the complement of $L(G,s,k)$ is a Hamilton cycle;

 \item [(d)] if $I(L(G,s,k))$ is a path of length at least $2$ or a cycle, then $k$ is odd;

   \item [(e)] if $\mathcal{D}$ is the collection of the unique minimal dangerous set in $G$ corresponding to the maximal independent sets of $L(G,s,k)$, then the union of the dangerous sets in $\mathcal{D}$ is $V(G)\setminus \{s\}$;

\item[(f)] if $I(L(G,s,k))$ is a path of length $1$, then the complement of $L(G,s,k)$ is two cliques of size $(\deg(s)+1)/2$ intersecting in a vertex ($L(G,s,k)$ is an isolated vertex plus a balanced complete bipartite graph);

 \item[(g)] if $I(L(G,s,k))$ is a path of length at least $2$, and $A_1,\ldots, A_m$ is the corresponding path of dangerous sets in $G$, then 
\begin{itemize} 
\item $|\delta(\{s\}:A_1)|=|\delta(\{s\}:A_m)|$; 
\item $|\delta(\{s\}:A_1)|\leq \frac{k+3}{2}$; and 
\item  $|\delta(A_1:A_m)|=|\delta(A_1\setminus A_2 : A_m\setminus A_{m-1})|=\frac{k+3}{2}-l$, where
\\ $l= |\delta(\{s\}:A_1)|=|\delta(\{s\}:A_m)|$;
\end{itemize}

\item [(h)] if $I(L(G,s,k))$ is a path, then the complement of $L(G,s,k)$ is two cliques of the same size, at most $(k+3)/2$, joined by a path.

 \end{enumerate}

\end{theorem}

\begin{proof}
\begin{par}
 By Lemma \ref{component of max indep sets} every component of $I(L(G,s,k))$ is either a path or a cycle. Thus, if $I(L(G,s,k))$ is connected, then it is either a path or a cycle. This proves $(a)$. 
 \end{par}
 \begin{par}
 By Lemma \ref{dangerous set between two other dangerous sets}, every maximal independent set in the cycle case has size $2$ and every maximal independent set that is represented by an interior vertex in the path case is also of size $2$. This gives $(b)$. From this, and Corollary \ref{component of complement} it follows that in case $I(L(G,s,k))$ is a cycle, the complement of $L(G,s,k)$ is a Hamilton cycle (thus we have $(c)$), and in case $I(L(G,s,k))$ is a path, then the complement of $L(G,s,k)$ consists of two cliques joined by a path (this gives part of $(h)$).
\end{par}
\begin{par}
Note than since $\deg(s)\geq 4$ and $s$ has exactly one neighbour in the intersection of two intersecting maximal independent sets, $I(L(G,s,k))$ is not a $3$-cycle. Thus, in case $I(L(G,s,k))$ is a path of length at least $2$ or a cycle, there are two maximal independent sets $I_1$ and $I_2$ in $L(G,s,k)$ such that $I_1\cap I_2\neq \emptyset$ and $I_1\cup I_2 \neq V(L(G,s,k))$. By Lemma \ref{two dangerous} it follows that $k$ is odd. This proves $(d)$.
\end{par}
\begin{par}
  The graph $I(L(G,s,k))$ consists of one component by assumption. This component is not a singleton because $L(G,s,k)$ consists of more than one maximal independent set by Frank's theorem \ref{Frank} since $\deg(s)>3$ and $G$ is $(s,k)$-edge-connected. Therefore, by Lemma \ref{vertices outside dangerous sets}, there are no vertices other than $s$ outside the union of the dangerous sets in $\mathcal{D}$. This proves $(e)$.
\end{par}

\begin{par}
  If $I(L(G,s,k))$ is a path, then it is not a singleton as explained above. Therefore, if $A_1,\cdots,A_m$ is the corresponding path of dangerous sets, then $m\geq 2$, and by Lemma \ref{vertices outside dangerous sets} any edges coming out of $A_1$ or $A_m$, other than the edges that connect them to $s$, $A_2$, or $A_{m-1}$, have their other ends in $A_m$ and $A_1$ respectively.
  \end{par}
  \begin{par}
  If $m=2$, then we have only two maximal independent sets intersecting in a vertex. Then by Theorem \ref{summary of lifting graph facts}, $L(G,s,k)$ is an isolated vertex plus a balanced complete bipartite graph, so its complement is two cliques of size $(\deg(s)+1)/2$ intersecting in a vertex. This proves $(f)$. By part $(6)$ in Lemma \ref{large independent}, $\deg(s)\leq k+2$ in that case. Thus the size of the two cliques is at most $(k+3)/2$. This proves $(h)$ in that case.
  \end{par}
  \begin{par}
  Now we suppose that $m\geq 3$ (i.e. $I(L(G,s,k))$ is a path of length at least $2$). Then there two intersecting maximal independent sets in $L(G,s,k)$ whose union is not $V(L(G,s,k))$. Thus by Lemma \ref{two dangerous}, $|\delta(A_1\cap A_2: A_2\setminus A_1)|=|\delta(A_{m-1}\cap A_m:A_{m-1}\setminus A_m)|= (k-1)/2$, and $|\delta_G(A_1)|=|\delta_G(A_m)|=k+1$. 
  \end{par}
  \begin{par}
  Thus, $|\delta(\{s\}:A_1)|= |\delta(\{s\}:A_m)|= (k+1)-(\frac{k-1}{2})-|\delta(A_1:A_m)|=(\frac{k+3}{2})-|\delta(A_1:A_m)|$. If we denote the number $|\delta(\{s\}:A_1)|= |\delta(\{s\}:A_m)|$ by $l$, then $l\leq (\frac{k+3}{2})$ as desired and $|\delta(A_1:A_m)|= (\frac{k+3}{2})-l$. This number, $l$, is the size of the two cliques in the complement of $L(G,s,k)$. This proves $(g)$ and completes the proof of $(h)$.
\end{par}\end{proof}

\begin{corollary}\label{the only disconnected case of lifting graph}
Let $G$ be an $(s,k)$-edge-connected graph such that $k\geq 2$ and $\deg(s)\geq 5$. The only case in which $L(G,s,k)$ is disconnected, is when it is an isolated vertex plus a complete balanced bipartite graph.
 \end{corollary}
 \begin{proof}
 \begin{par}
   If $I(L(G,s,k))$ is disconnected, then the complement of $L(G,s,k)$ is disconnected, so $L(G,s,k)$ is connected in that case. Thus we may assume that $I(L(G,s,k))$ is connected. Then, by Theorem \ref{connected complement}, it is a path or a cycle. If $L(G,s,k)$ is not an isolated vertex plus a complete bipartite graph ($I(L(G,s,k))$ is a path of length $1$), then $I(L(G,s,k))$ is either a path of length at least $2$ or a cycle.   
 \end{par}
 \begin{par}
   Thus the complement of the lifting graph is either a Hamilton cycle or two cliques of the same size with a path between them of length at least $1$.
\end{par}
\begin{par}
  If the complement of $L(G,s,k)$ is a cycle, then $L(G,s,k)$ is connected, since it contains at least $5$ vertices as $\deg(s)\geq 5$ by assumption.
  \end{par}
  \begin{par}
  If the complement is two cliques with a path between them and the path is of length $1$, then $L(G,s,k)$ is in fact a complete bipartite graph minus an edge, which is connected since the two cliques are of the same size which is at least $2$. 
  \end{par}
  \begin{par}
  If the path is of length at least $2$, then first and last maximal independent sets in the path, whose complements are the two cliques, and which have size at least $2$ each, induce a complete bipartite graph in $L(G,s,k)$. Also any vertex outside the two cliques is a neighbour to all the vertices in the two cliques except possibly one vertex from each clique. Thus, $L(G,s,k)$ is connected in that case too, in fact there is a path of length at most $2$ between any two vertices in it.
\end{par}
\end{proof}

\section{Structure of a graph from the structure of its lifting graph}

\begin{par}
The results of the previous sections about maximal independent sets in the lifting graph and their corresponding dangerous sets in $G$ give us the following theorem for the structure of $G$. When the complement of the lifting graph is connected, the graph $G$ has a path or cycle structure of dangerous sets, illustrated in Figures \ref{cycle case} and \ref{path case}. Note that the vertex set of each blob in the figures is not a dangerous set, it is the intersection of two dangerous sets, except for the first and last blobs in Figure \ref{path case}, which have vertex sets $A_1\setminus A_2$ and $A_m\setminus A_{m-1}$.
\end{par}
\begin{par}
If $\deg(s)>4$, the only case in which $L(G,s,k)$ is disconnected is when it is an isolated vertex plus a balanced complete bipartite graph (Lemma \ref{the only disconnected case of lifting graph}). In that case its complement is two cliques of size $(\deg(s)+1)/2$ intersecting in one vertex, and $I(L(G,s,k))$ is a path of length $1$. If $A_1$ and $A_2$ are the two dangerous sets corresponding to the two maximal independent sets of $L(G,s,k)$, then $A_1\setminus A_2$, $A_1\cap A_2$, and $A_2\setminus A_1$ are the vertex sets of the blobs of the $(s,r)$-path structure of $G$, where $r\geq (k-1)/2$.
\end{par}
\begin{par}
When $\deg(s)=4$, there is one case where $L(G,s,k)$ is disconnected. This is the case when $L(G,s,k)$ is a perfect matching ($K_2\cup K_2$) and its complement is a $4$-cycle. By Lemma \ref{4 cycle}, $G$ has an $(s,(k-1)/2)$-cycle structure in that case. See Figure \ref{degree_4_cycle}.
\end{par}
\begin{par}
Other than these two cases, if the complement of $L(G,s,k)$ is connected, then so is $L(G,s,k)$. Thus, in the following theorem we assume that $\deg(s)>4$ and that both $L(G,s,k)$ and its complement are connected. Recall from part $(d)$ in Theorem \ref{connected complement} that this situation happens only when $k$ is odd, and from part $(a)$ that $I(L(G,s,k))$ is either a path or a cycle, and we also know from Lemma \ref{two dangerous} the exact value of $r$ in the path or cycle structure in that case, it is $(k-1)/2$. Note also that in that case if $A_1,\cdots, A_m$ is the path or cycle of dangerous sets corresponding to $I(L(G,s,k))$, then $m\geq 3$, in particular if $I(L(G,s,k))$ is a path, it is a path of length at least $2$.
\end{par}

\begin{figure}[!ht]
  \centering
  \begin{minipage}[b]{0.6\textwidth}
    \includegraphics[width=\textwidth]{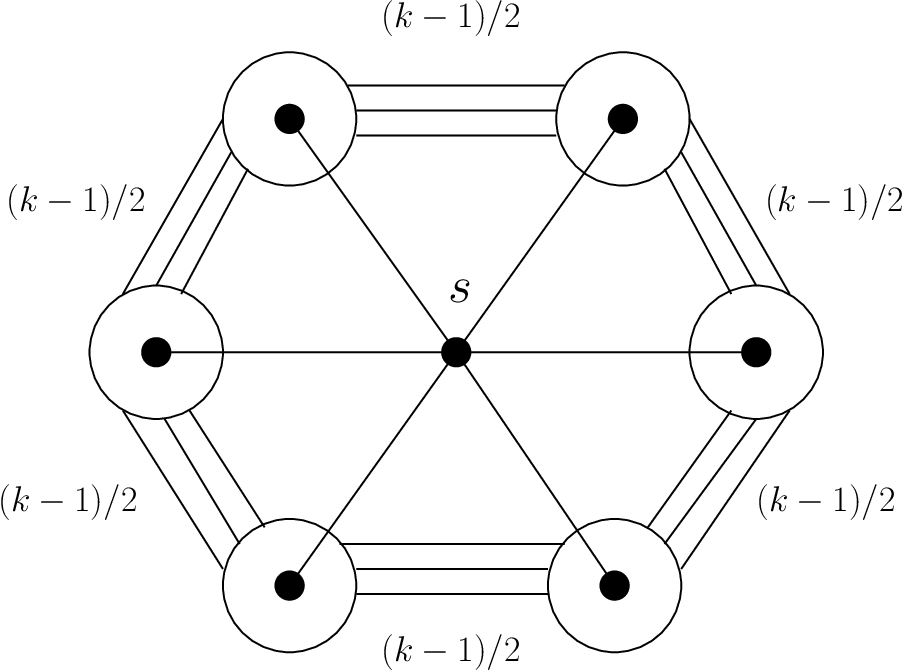}
    \caption{Cycle of intersections of dangerous sets.}
    \label{cycle case}
  \end{minipage}
  \hfill
  \begin{minipage}[b]{0.6\textwidth}
    \includegraphics[width=\textwidth]{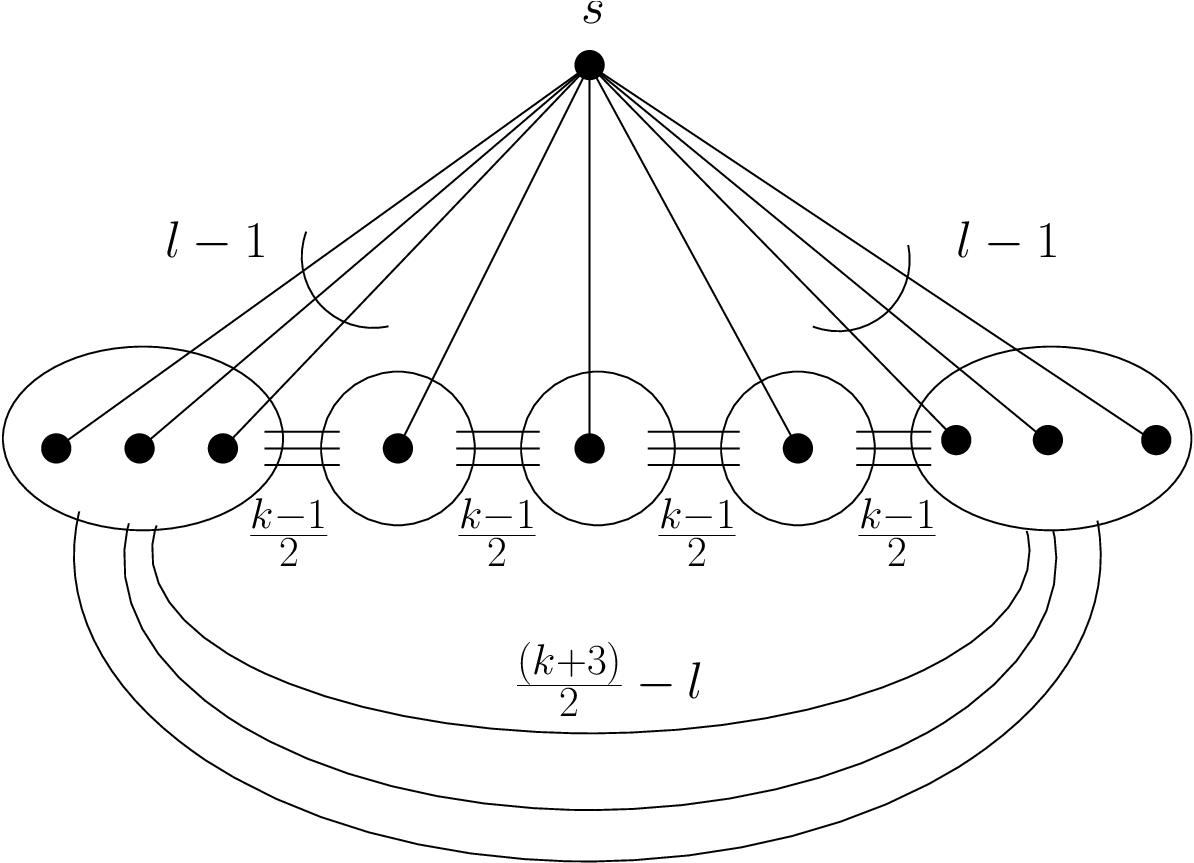}
    \caption{Path of intersections of dangerous sets.}
    \label{path case}
  \end{minipage}
\end{figure}

Point (i) below is comparable to \cite[Theorem~3.6]{Jordan1999ConstrainedSplitting} where it is proved, under the assumption that $\deg(s)$ is even, that the complement of the lifting graph is a cycle if and only if $G$ is \emph{round} (has a cycle structure) with respect to $s$. The parameter of $(k-1)/2$ was also found in \cite[Lemma~2.3~(ii)]{Jordan1999ConstrainedSplitting} for two crossing maximal dangerous sets when $\deg(s)$ is even.

\begin{theorem}\label{cycle or path structure of G}
Let $G$ be an $(s,k)$-edge-connected graph such that $k\geq 2$ and $\deg(s)> 4$. Suppose that $L(G,s,k)$ and its complement are both connected, and let $A_1,\cdots, A_m$ be the path or the cycle of dangerous sets corresponding to $I(L(G,s,k))$. We have the following:
\begin{enumerate}
    \item[(i)] \textbf{Cycle case:} If $I(L(G,s,k))$ is a cycle, then $G$ has an $(s,(k-1)/2)$-cycle structure whose blobs the subgraphs induced by $A_1\cap A_m$ and by $A_i\cap A_{i+1}$, for $i\in \{1,\cdots, m-1\}$.
    
    \item[(ii)] \textbf{Path case:} If $I(L(G,s,k))$ is a path, then $G$ has an $(s,(k-1)/2)$-path structure whose blobs are the subgraphs induced by $A_i\cap A_{i+1}$ for $i\in \{1,\cdots, m-1\}$, $A_1\setminus A_2$, and $A_m\setminus A_{m-1}$.
 \end{enumerate}   

\end{theorem}

\begin{proof}
\begin{par}
  By part $(e)$ in Theorem \ref{connected complement}, $V(G)= \{s\} \cup A_1\cup \cdots A_m$. By Lemma \ref{two dangerous} there is exactly one neighbour of $s$ in each intersection $A_i\cap A_{i+1}$, for $i\in \{1,\cdots, m-1\}$, and in the cycle case one neighbour in $A_m\cap A_1$. By points 
$(4)$ and $(5)$ of Lemma \ref{dangerous set between two other dangerous sets}, $A_i\setminus (A_{i-1}\cup A_{i+1})= \emptyset$ for each $i\in \{2,\cdots, m-1\}$, and in the cycle case also $A_1\setminus (A_2\cup A_m)=\emptyset$ and $A_m\setminus (A_{m-1}\cup A_1)= \emptyset$. Also there are $(k-1)/2$ edges between $A_i\cap A_{i+1}$ and $A_{i+1}\cap A_{i+2}$, for each $i\in \{1,\cdots, m-2\}$, and in the cycle case also between $A_{m-1}\cap A_m$ and $A_m\cap A_1$, and between $A_m\cap A_1$ and $A_1\cap A_2$. Moreover, by Lemma \ref{two dangerous} the only edges coming out of the intersections are the edges to $s$ and to the the two dangerous sets for which this is the intersection. These are the edges we have identified.
\end{par}\end{proof}

We have a few further remarks on these structures. Recall that the path an cycle structures occur also when $L(G,s,k)$ is an isolated vertex plus a balanced complete bipartite graph, and when $\deg(s)=4$. We therefore assume in the following corollary only that $I(L(G,s,k))$ is connected and $\deg(s)\geq 4$, a small change from the assumptions of the theorem.

\begin{corollary} \label{the blobs are (k+1)/2 connected}
Let $G$ be an $(s,k)$-edge-connected graph such that $k\geq 2$ and $\deg(s)\geq 4$, and assume that $I(L(G,s,k))$ is connected. Then:
\begin{enumerate}
\item [(1)] For every blob $B$ of the path or cycle structure of $G$, $|\delta(B)|=k$, except possibly in the case if $B$ is the middle blob when $L(G,s,k)$ is an isolated vertex plus a balanced complete bipartite graph, where in that case $|\delta(B)|\geq k$.
\item[(2)] For every vertex $v$ in a blob $B$, there are $k$ edge-disjoint paths from $v$ ending with $k$ distinct edges in $\delta(B)$.
\item [(3)] Each blob is $\lceil k/2 \rceil$-edge-connected except possibly for the middle blob when $L(G,s,k)$ is an isolated vertex plus a balanced complete bipartite graph.
\end{enumerate}
\end{corollary}
\begin{proof}
\begin{par}
There are exactly $k$ edges incident with each blob in the cycle case by Theorem \ref{cycle or path structure of G}, part $(i)$, for $\deg(s)>4$, and by Lemma \ref{4 cycle} for $\deg(s)=4$. In the path case, if $L(G,s,k)$ is connected, then by Theorem \ref{cycle or path structure of G}, part $(ii)$, there are exactly $k$ edges incident with each blob other than the first an the last ones, and there are $(k-1)/2$ edges from the last blob to the one before it, and $(k-1)/2$ edges from the first blob to the one after it. For each of the first and the last blobs, Theorem \ref{connected complement}, part $(g)$, gives us another $(k+1)/2$ edges ($\frac{(k+3)}{2}-l+l-1$ where the subtracted $1$ is for the edge between $s$ and either the second blob or the one before the last). Thus, also each of the first and last blobs have $k$ edges incident with them in the path case when $L(G,s,k)$ is connected. 
\end{par}
\begin{par}
Now it remains to consider the case when $L(G,s,k)$ is an isolated vertex plus a balanced complete bipartite graph. By parts $(3)$ and $(4)$ of Lemma \ref{large independent}, and since the middle blob contains exactly one neighbour of $s$, it follows that with each of the first and last blobs ($A_1\setminus A_2$ and $A_2\setminus A_1$) there are $k$ incident edges, and at least $k$ edges incident with the middle blob (which is the intersection $A_1\cap A_2$).
\end{par} 
\begin{par}
Let $v$ be a vertex in a blob $B$ that is not the middle blob in the case when $L(G,s,k)$ is an isolated vertex plus a balanced complete bipartite graph. Let $w\neq s$ be any vertex outside the blob. Since $G$ is $(s,k)$-edge-connected, there are $k$ edge-disjoint paths between $v$ and $w$. Each one of those paths contains at least one edge of $\delta(B)$ since $w$ is outside $B$. If $|\delta(B)|=k$, then each one of the paths has exactly one edge of $\delta(B)$.
\end{par}
\begin{par}
Consider any two vertices in a blob $B$ such that $|\delta(B)|=k$. There are $k$ edge-disjoint paths between them and at most $\lfloor k/2 \rfloor$ of those paths can go out and back into $B$ since $|\delta(B)|=k$. Thus there are at least $\lceil k/2 \rceil$ edge-disjoint paths between them inside $B$, so $B$ is $\lceil k/2 \rceil$-edge-connected.
\end{par}\end{proof}

\begin{figure}[!h]
  \centering
  \begin{minipage}[b]{0.5\textwidth}
    \includegraphics[width=\textwidth]{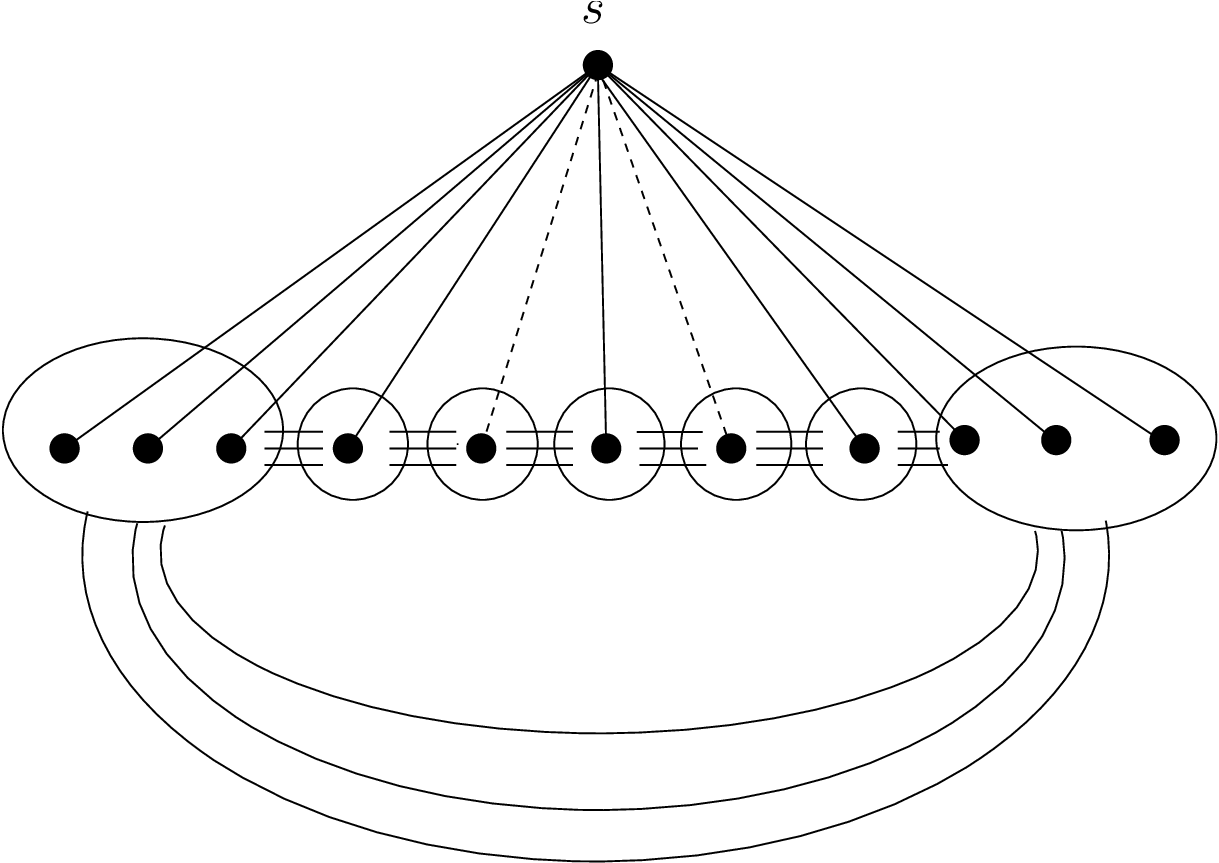}
  \end{minipage}
  \hfill
  
 \vspace{0.1in}
  
  \begin{minipage}[b]{0.5\textwidth}
    \includegraphics[width=\textwidth]{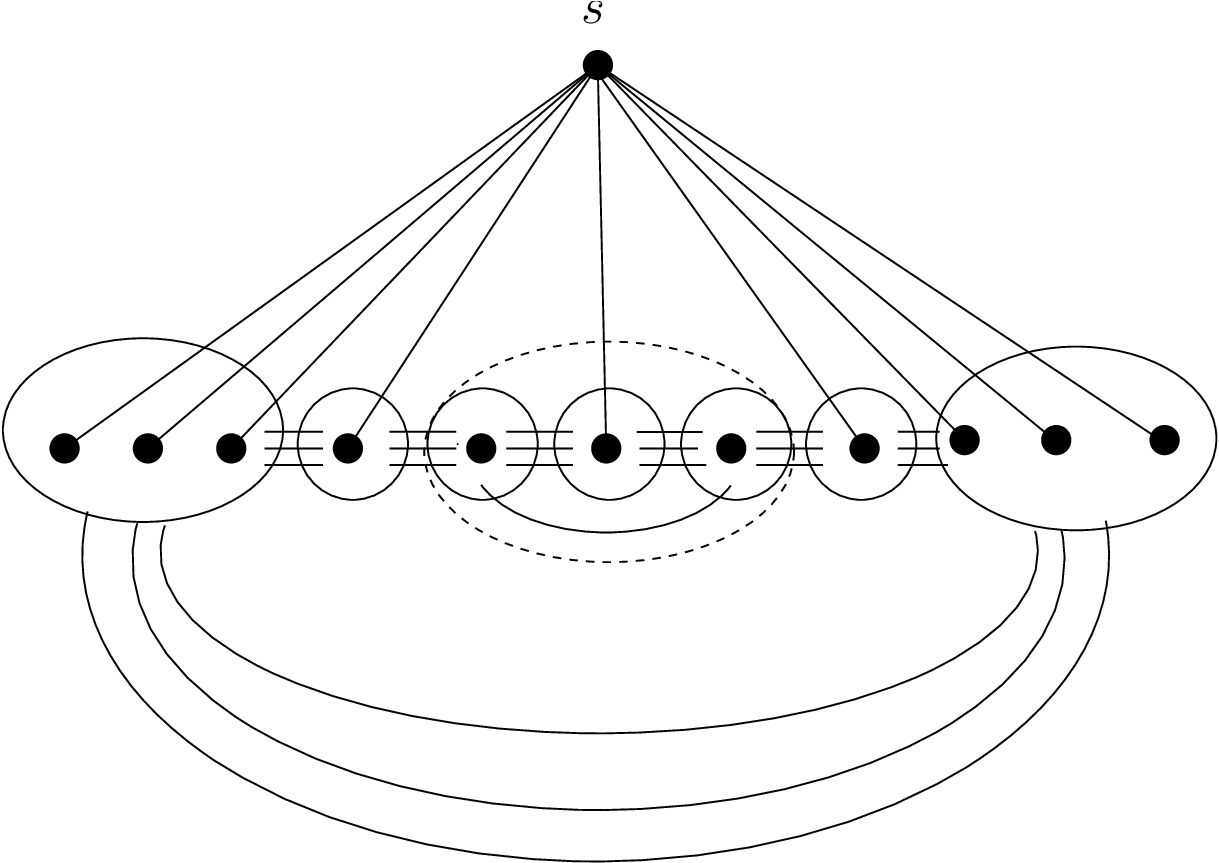}
    \end{minipage}

\vspace{0.1in}
  
  \begin{minipage}[b]{0.5\textwidth}
    \includegraphics[width=\textwidth]{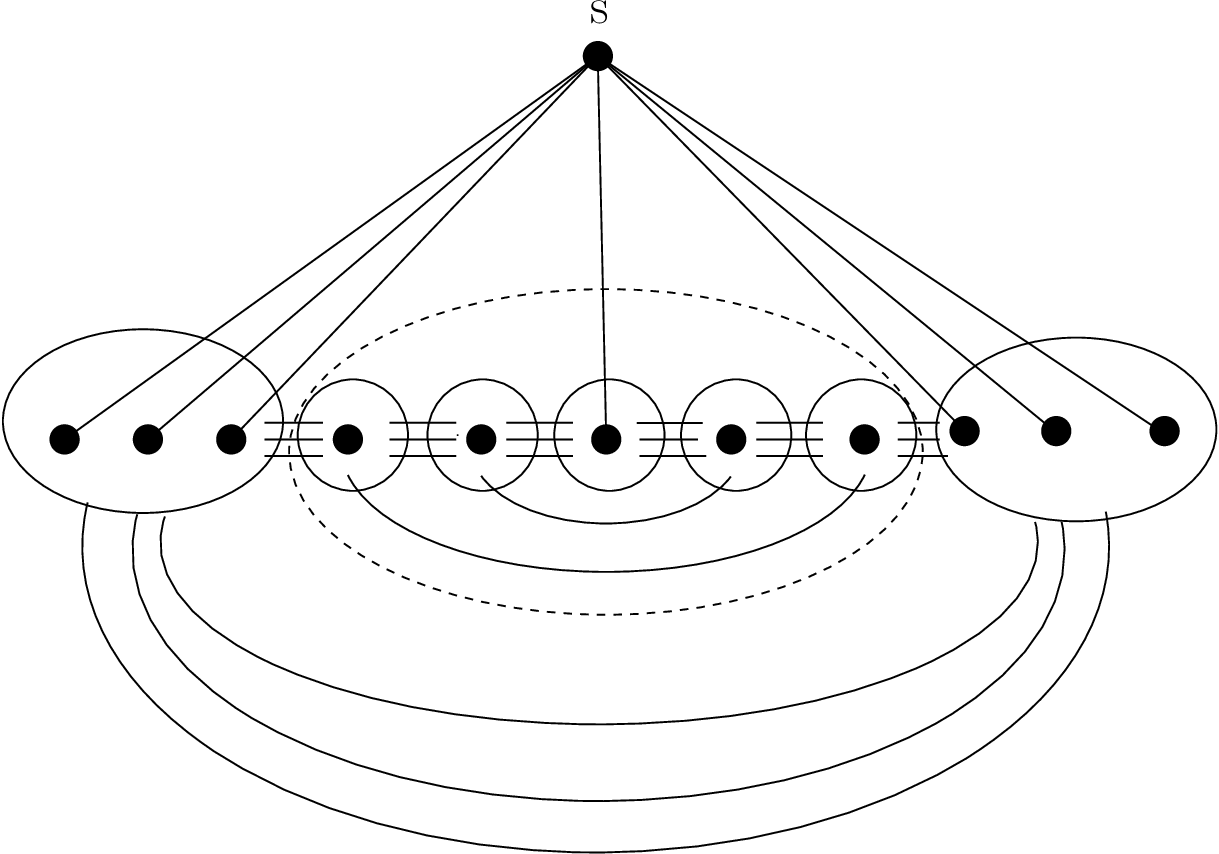}
    \end{minipage}
    \caption{Successively lifting pairs of edges until we reach a lifting graph that is an isolated vertex plus a complete bipartite graph.}
   \label{Figure of successive lifting in the path case}
\end{figure}

\begin{remark}\label{successive lifting in the path case}
Suppose that $\deg(s)$ is odd and $I(L(G,s,k))$ is a path that is not a singleton. Let $B_1,\cdots, B_n$ be in order the blobs of the corresponding path structure in $G$, and for each $i\in \{2,\cdots, n-1\}$ let $s_i$ be the unique neighbour of $s$ in $B_i$. The blob $B_{(n+1)/2}$ is the middle blob, and $B_{(n-1)/2}$ and $B_{(n+3)/2}$ are the ones before and after it respectively.
Successively lifting pairs of edges from the inside out brings us, at some step, to the isolated vertex plus balanced complete bipartite case, that is lifting in the following order $(ss_{(n-1)/2 }, ss_{(n+3)/2 }), (ss_{(n-3)/2 }, ss_{(n+5)/2 }) \ldots$ until only the middle, first, and last blobs are left. This is illustrated in Figure \ref{Figure of successive lifting in the path case}. Throughout, the path structure is preserved until a graph is reached whose lifting graph at $s$ is an isolated vertex plus $K_{(l-1),(l-1)}$, where $l$ is the number of edges between $s$ and each of the first and last dangerous sets in the path structure (so $l-1$ edges to each of the first and last blobs).
\end{remark}

\begin{remark}\label{lifting in the isolated vertex plus bipartite case}
  Lifting a pair of edges from two different sides in case the lifting graph is an isolated vertex plus a balanced complete bipartite graph, results in a graph whose lifting graph is also an isolated vertex - the same vertex - plus a balanced complete bipartite graph. Also when the lifting graph is complete multipartite, its multipartite structure is kept through lifting. This can be seen by looking at the dangerous sets corresponding to the maximal independent sets of the lifting graph, which are pairwise disjoint in this case, and for each one, $A$, of them, the size of the cut $\delta(A)$ stays the same after lifting.
\end{remark}

\begin{remark}\label{lifting in the cycle case}
  Suppose that $\deg(s)\geq 4$ and that $I(L(G,s,k))$ is a cycle. Let $B_0, \cdots, B_{n-1}$ be the blobs of the corresponding $(s,(k-1)/2)$-cycle structure in $G$. Lifting a pair of edges incident with $s$ that have their non-$s$ ends in $B_{i-1}$ and $B_{i+1}$ for some $i$ (indices modulo $n$) gives an $(s, (k-1)/2)$-cycle structure again on a cycle of length less by $2$. The union of $B_{i-1}$, $B_i$, and $B_{i+1}$ gives one new blob. All the other blobs remain the same.
\end{remark}

\section{Degree Three}

Degree three needs a special treatment as it is possible that no pair of edges incident with $s$ is liftable when $\deg(s)=3$. Here we show that if this is the case, then $G$ has a specific structure. This is presented in the following proposition. Part $(1)$ is the case when $\deg(s)=3$ and $I(L(G,s,k))$ is a cycle. Figure \ref{deg(s)=3} shows examples for the situation described in $(2)$.

\begin{figure}
\centering
   \includegraphics[scale=0.8]{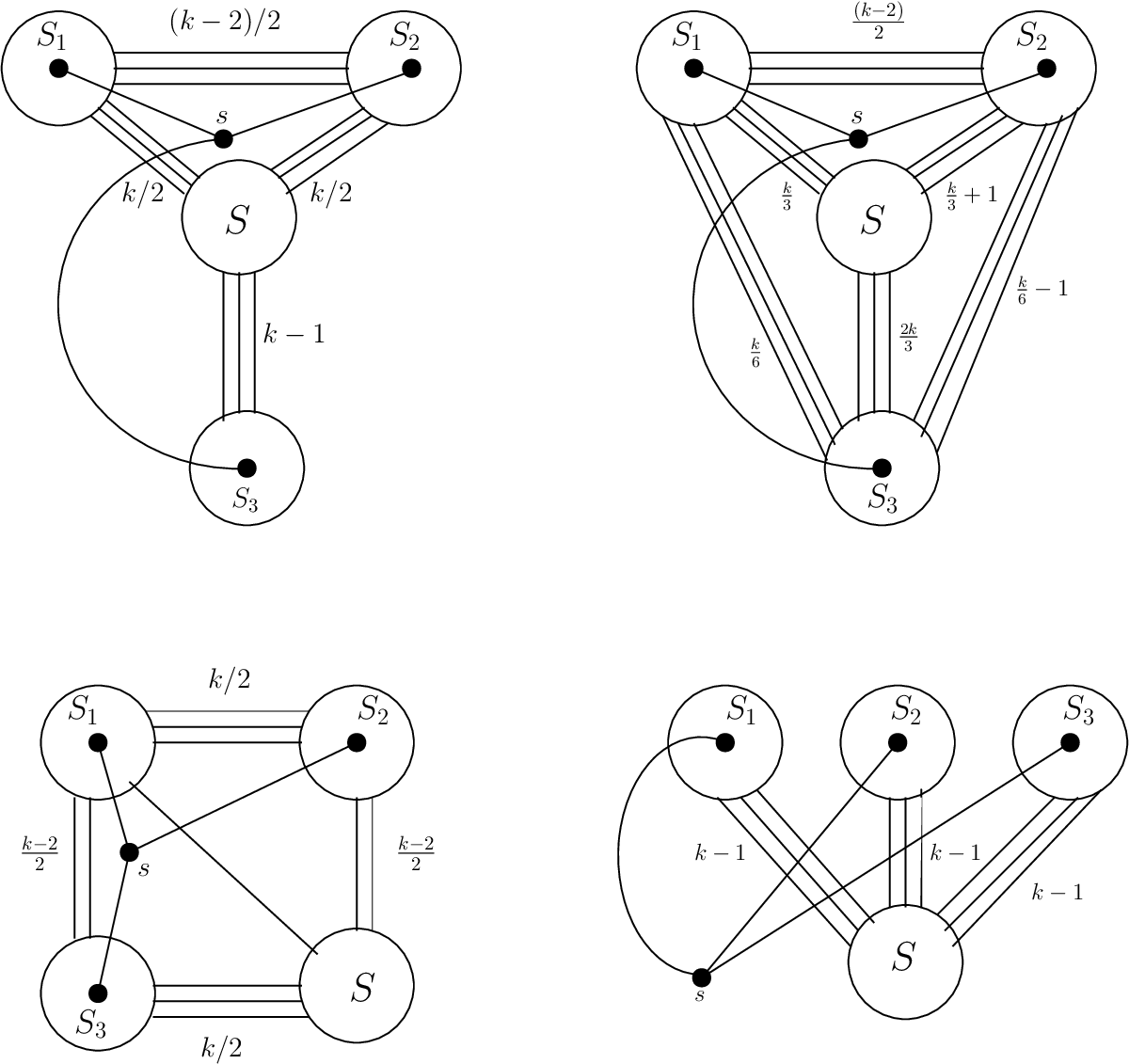} 
   \caption{Examples of graphs with $deg(s)=3$, that do not have an $(s, (k-1)/2)$-cycle structure, where no pair of edges incident with $s$ is liftable.}
    \label{deg(s)=3}
\end{figure}

\begin{proposition}\label{structure of degree three}
 Let $G$ be an $(s,k)$-edge-connected graph such that $k\geq 2$ and $\deg(s)=3$. If no pair of edges incident with $s$ is liftable, then $s$ has three distinct neighbours (no parallel edges) $s_1$, $s_2$, $s_3$, and there are pairwise disjoint sets of vertices $S_1$, $S_2$, and $S_3$ in $G-s$ such that for each $i\in \{1,2,3\}$, $S_i\cap \{s_1,s_2,s_3\}=\{s_i\}$, and

\begin{enumerate}

\item[(1)] either $G$ has an $(s, (k-1)/2)$-cycle structure with blobs $S_1$, $S_2$, and $S_3$, or

\item[(2)] $S_1\cup S_2\cup S_3\cup\{s\}\ne V(G)$ and for each $i\in \{1,2,3\}$, $|\delta_G(S_i)|=k$. 

\end{enumerate}

In the first case, for every distinct $i$ and $j$ in $\{1,2,3\}$, the dangerous set corresponding to $\{ss_i, ss_j\}$ is $S_i \cup S_j$, and in the second case it is $S_i \cup S \cup S_j$, where $S=V(G)\setminus (S_1\cup S_2\cup S_3\cup\{s\})$. Moreover, in the second case, the subgraphs induced by the intersections of the dangerous sets, $S_1\cup S$, $S_2\cup S$, $S_3\cup S$, each contain a cut of size at most $k-1$, namely, the set of edges between $S$ and $S_i$ for $i\in\{1,2,3\}$. See Figure \ref{deg(s)=3}.
\end{proposition}

\begin{proof}
\begin{par}
If all edges incident with $s$ are parallel, then clearly every pair of them is liftable, because in this case any path between two vertices different from $s$ does not need to go through $s$. Assume now that $s$ has exactly two neigbours, $v$ and $w$, where $sv$ has multiplicity $2$. Assume also that no pair of edges incident with $s$ is liftable. By Lemma \ref{dangerous set for degree three} there is a dangerous set $A$ containing $v$ and $w$ since one of the edges $sv$ is not liftable with $sw$, but then $A$ contains all the neighbours of $s$, contradicting Lemma \ref{no dangerous set with all three neighbours}. Thus the neighbours of $s$ are distinct if no pair of edges is liftable. Assume that this is the case and that $s_1$, $s_2$, and $s_3$ are the neighbours of $s$.
\end{par}
\begin{par}
Again by Lemmas \ref{dangerous set for degree three} and \ref{no dangerous set with all three neighbours}, for distinct $i,j\in \{1,2,3\}$ there is a dangerous set $A_{i,j}=A_{j,i}$ such that $A_{i,j}$ contains $s_i$ and $s_j$ but not $s_l$ for $l\in \{1,2,3\}\setminus \{i,j\}$. 
\end{par}
\begin{par}
  By Lemma \ref{large independent}, noting that it is valid for $\deg(s)=3$, we have the following for distinct $i,j,l\in\{1,2,3\}$,
  \begin{enumerate}
  \item[(a)] there are no vertices outside $A_{i,l}\cup A_{j,l}$ other than $s$; and
  \item[(b)]$|\delta(A_{j,l}\setminus A_{i,l}: A_{i,l}\setminus A_{j,l})|\leq (k-1)/2$. 
 \end{enumerate}
\end{par}
\begin{par}
  Define $S_1=A_{1,3}\setminus A_{2,3}$, $S_2=A_{1,2}\setminus A_{1,3}$, and $S_3=A_{2,3}\setminus A_{1,2}$. Then $S_1$, $S_2$, and $S_3$ are pairwise disjoint, and the only neighbour of $s$ contained in $S_i$ is $s_i$, for $1\leq i \leq 3$. By (a), $S_1\subseteq A_{2,3}\cup A_{1,2}$, but since $S_1$ is in $\bar{A}_{2,3}$, then $S_1\subseteq A_{1,2}\setminus A_{2,3}$. This by $(b)$ implies that $|\delta(S_1:S_3)|\leq (k-1)/2$. Similarly, $S_2\subseteq A_{2,3}\setminus A_{1,3}$, so $|\delta(S_1:S_2)|\leq (k-1)/2$, and $S_3\subseteq A_{1,3}\setminus A_{1,2}$, so $|\delta(S_2:S_3)|\leq (k-1)/2$. If $S_1\cup S_2\cup S_3\cup\{s\}=V(G)$, then each inequality has to be an equality since $G$ is $(s,k)$-edge-connected.
\end{par}
\begin{par}
  If $S_1\cup S_2\cup S_3\cup\{s\}\neq V(G)$, then let $S=V(G)\setminus (S_1\cup S_2\cup S_3\cup\{s\})$. Note that 
$S$ is also equal to $A_{2,3} \cap A_{1,3}\cap A_{1,2}$. Thus, $S\subseteq (A_{i,l}\cap A_{j,l})$ for every distinct $i,j,l\in\{1,2,3\}$. By their definitions, $S_1$ is disjoint from $A_{2,3}$, $S_2$ is disjoint from $A_{1,3}$, and $S_3$ is disjoint from $A_{1,2}$. From the previous paragraph we have that $A_{1,2}$ contains $S_2$ and $S_1$. It also contains $S=A_{2,3} \cap A_{1,3}\cap A_{1,2}$, and since it is disjoint from $S_3$, it follows that $A_{1,2}=S_1\cup S \cup S_2$. Similarly, $A_{1,3}=S_1\cup S \cup S_3$, and $A_{2,3}= S_2 \cup S \cup S_3$. Thus, $A_{1,2} \cap A_{1,3}= S_1 \cup S$, $A_{1,3}\cap A_{2,3}= S_3 \cup S$, and $A_{1,2}\cap A_{2,3}= S_2 \cup S$. See Figure \ref{deg(s)=3}.
 \end{par}
 \begin{par}
   Thus, for $\{i,j,l\}=\{1,2,3\}$, $|\delta(A_{i,j})|= 2+|\delta(S_i:S_l)|+|\delta(S:S_l)|+|\delta(S_j:S_l)|=2+|\delta_{G-s}(S_l)|$, so $|\delta_{G-s}(S_l)|\leq k-1$ since $A_{i,j}$ is dangerous. Recall that $S_l$ contains exactly one neighbour of $s$, and that $G$ is $(s,k)$-edge-connected. Therefore $\delta(S_l)=k$ as desired. Note that $|\delta(S:S_l)|\leq k-1$ for every $l\in \{1,2,3\}$.
 \end{par}\end{proof}

\begin{figure}[!h]
  \centering
  \begin{minipage}[b]{0.45\textwidth}
    \includegraphics[width=\textwidth]{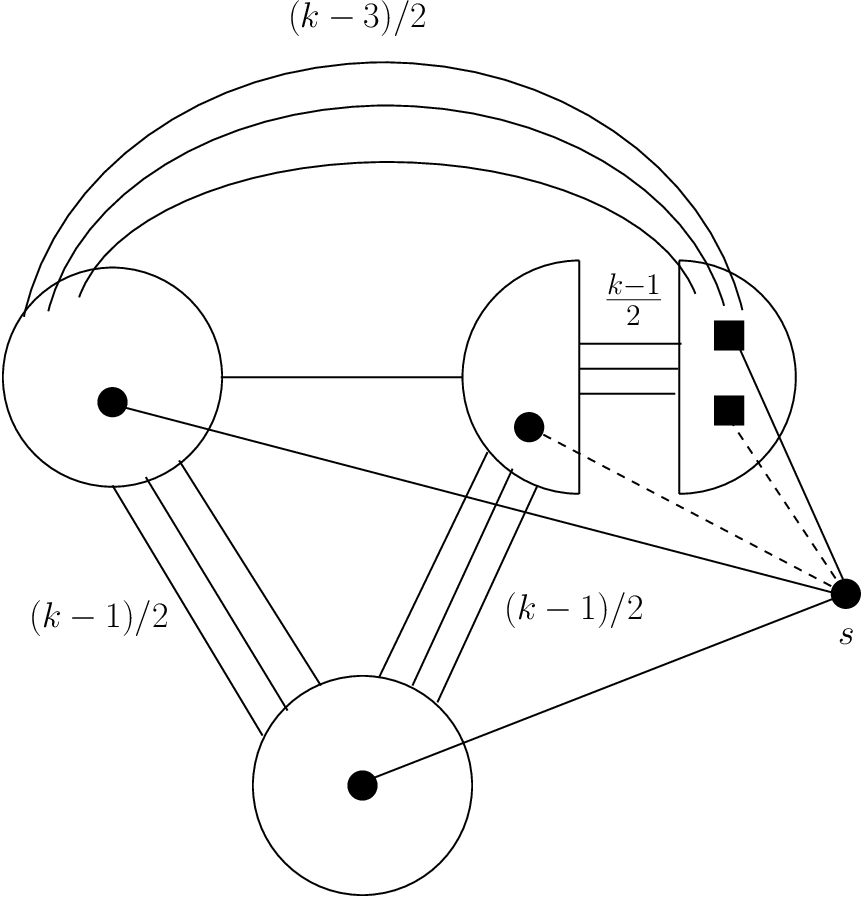}
  \end{minipage}
  \hfill
  \begin{minipage}[b]{0.45\textwidth}
    \includegraphics[width=\textwidth]{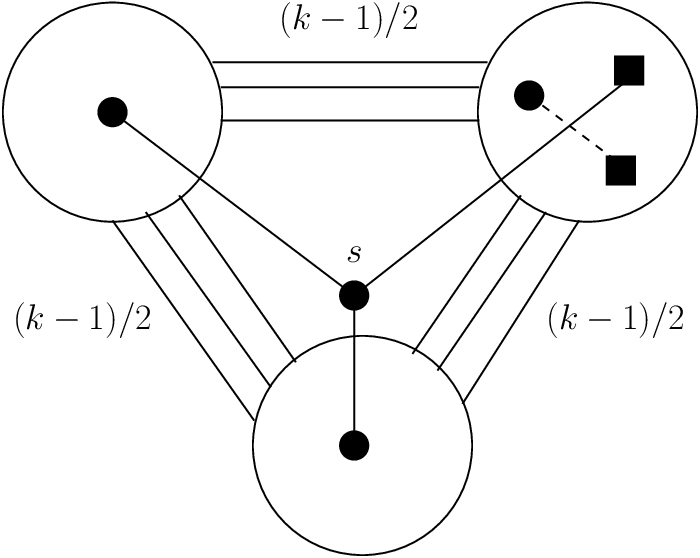}
    \end{minipage}
 \caption{A graph with $deg(s)=5$ whose lifting graph is $K_{2,3}$ such that when the dashed pair of edges is lifted no remaining pair of edges incident with $s$ is liftable.}
    \label{example 1 lifting from deg 5 to deg 3}
\end{figure}

\begin{figure}[!h]
  \centering
  \begin{minipage}[b]{0.45\textwidth}
    \includegraphics[width=\textwidth]{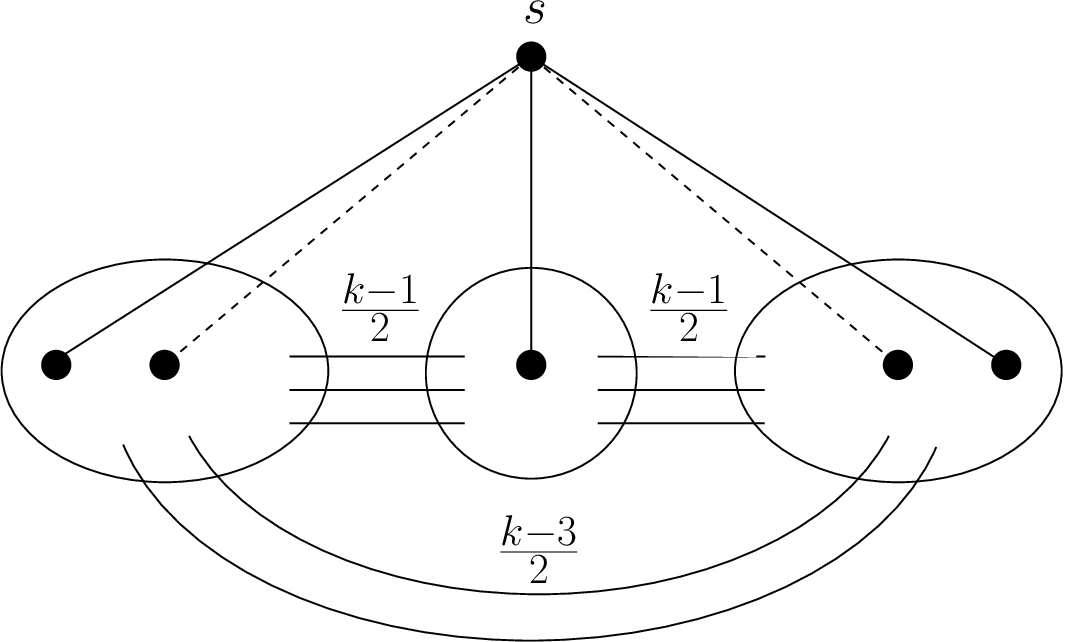}
  \end{minipage}
  \hfill
  \begin{minipage}[b]{0.45\textwidth}
    \includegraphics[width=\textwidth]{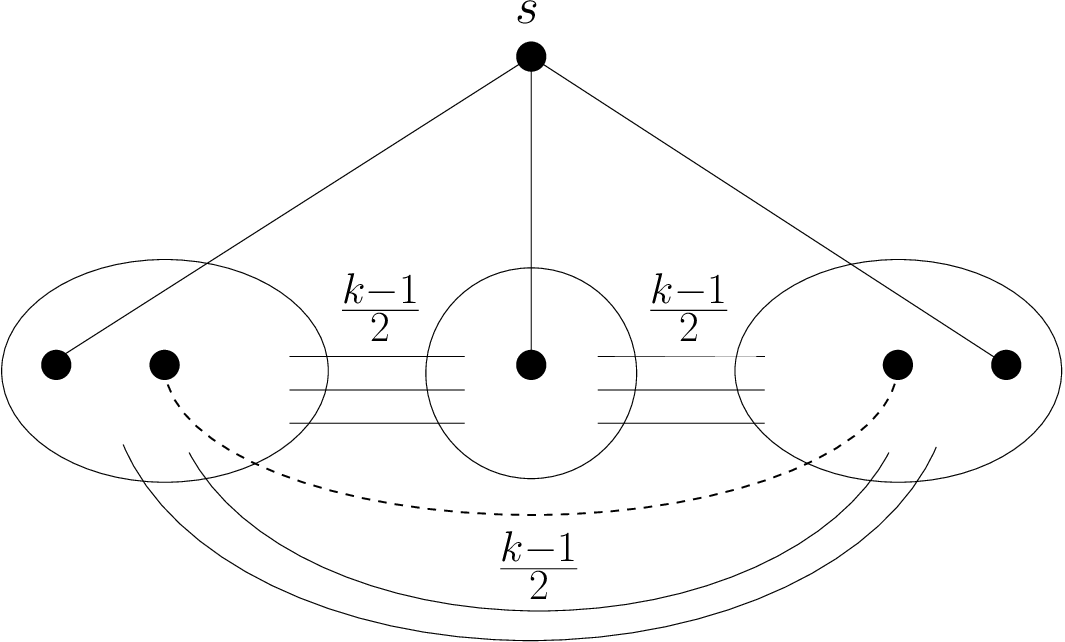}
    \end{minipage}
 \caption{A graph with $deg(s)=5$ whose lifting graph is an isolated vertex plus $K_{2,2}$ such that when any pair of edges on two different sides is lifted, no remaining pair of edges incident with $s$ is liftable, as any second lift will result in a cut of size $k-1$.}
    \label{example 2 lifting from deg 5 to deg 3}
\end{figure}

\begin{remark}
Note that when all three edges incident with $s$ are pairwise non-liftable it is possible that $k$ is even as well as odd.
\end{remark}

\begin{remark}
The situation when $deg(s)=3$ and no pair of edges incident with it is liftable, can come from lifting a pair of edges in graphs $G$ with different lifting graph structures for $deg(s)=5$ as illustrated in Figures \ref{example 1 lifting from deg 5 to deg 3} and \ref{example 2 lifting from deg 5 to deg 3}.
\end{remark}

\begin{remark}
When $k=1$, $\deg(s)=3$, then no pair of edges incident with $s$ is liftable if and only if each edges incident with $s$ is a cut-edge.
\end{remark}

\vspace{0.1in}

\noindent\textbf{\large Acknowledgement:} The author thanks her PhD. supervisor Bruce Richter for his careful reading of the paper and for the many helpful and detailed comments.

\bibliographystyle{plain}
\bibliography{reference}

\end{document}